\documentclass[ onefignum,onetabnum]{siamart251216}
\usepackage{amsmath, latexsym}
\usepackage{amssymb}
\usepackage[utf8]{inputenc}
\usepackage[english]{babel}
\usepackage{tikz, graphicx}
\usetikzlibrary{shapes.geometric, arrows}
\usetikzlibrary{positioning}
\usepackage{textpos}
\usepackage{amsfonts}
\usepackage{newlfont}
\usepackage{enumerate}
\usepackage{epsfig}
\usepackage{subfig}
\usepackage{subcaption}
\usepackage{float} 
\usepackage{multirow}
\usepackage{ntheorem}
\usepackage{mathrsfs}
\usepackage{bm}

\newsiamremark{remark}{Remark}

\def\T{T\wedge\tau_N^n}
\def\L{\mathcal{L}}

\def\I{\mathcal{I}}
\def\PP{\mathscr{P}}
\def\F{\mathcal{F}}

\def\f{\mathbf{f}}

\def\B{\mathcal{B}}
\def\BB{\mathfrak{B}}
\def\D{\mathcal{D}}

\def\Y{\mathrm{Y}}

\def\G{\mathcal{G}}
\def\R{\mathcal{R}}
\def\T{\mathcal{T}}

\def\X{\mathbb{X}}

\def\U{\mathcal{U}}
\def\X{\mathcal{X}}
\def\Y{\mathcal{Y}}

\def\PR{\mathcal{PR}}
\def\I{\mathcal{I}}

\def\N{\mathcal{N}}

\def\f{\boldsymbol{f}}
\def\g{\boldsymbol{g}}

\newcommand{\curl}{\textbf{curl}}

\ifpdf
\hypersetup{
	pdftitle={An Example Article},
	pdfauthor={h}
}
\fi

\author{Shiv Mishra\thanks{``Department of Mathematics, Indian Institute of Technology Roorkee, Roorkee 247667, India."\\
		E-mail: {shiv\_m@ma.iitr.ac.in}, {arbaz@ma.iitr.ac.in}. \\
		\textbf{Funding:} ``SM was supported by University Grants Commission (UGC), Govt. of India (File no.: 393/2023/221610069248). AK was supported by ANRF ARG-Matrics  ANRF/ARGM/2025/001949/MTR."}
	\and
	Arbaz Khan\footnotemark[1]}
\ifpdf
\hypersetup{pdftitle={New \LaTeX\ Style} ,pdfauthor={S. Mishra and A. Khan}}
\fi

\begin{document}
	\title{Structure-Preserving and Pressure-Robust  PINNs for Incompressible Oseen Problems}
\date{Date: \today}
\maketitle

\begin{abstract}
We develop a new class of physics-informed neural network approximations for the stationary Oseen equations based on stability-consistent loss constructions. In contrast to standard PINN formulations, which are typically heuristic, the proposed consistent PINN (CPINN) framework is systematically derived from the stability structure of the continuous problem. Within this setting, we introduce two fundamentally new approaches. First, we design standard CPINN formulations that exhibit clear improvements over conventional PINNs. Second, we propose pressure-robust CPINN formulations that provably eliminate the influence of gradient forces on the velocity approximation, yielding velocity errors that depend solely on the divergence-free component of the forcing and are independent of the pressure.
The framework accommodates both exactly divergence-free architectures and unconstrained velocity approximations, providing a unified treatment of these two paradigms. Using techniques from optimal recovery theory, we establish, for the first time in the PINN setting for Oseen-type problems, quantitative recovery estimates and optimal error bounds for both velocity and pressure under suitable Besov regularity assumptions. In particular, we obtain optimal rates for the velocity in $\boldsymbol{H}^1(\Omega)$ and for the pressure in $L^2(\Omega)$.
The proposed methodology introduces a pressure-robust CPINN paradigm for incompressible flows, combining structural consistency, robustness with respect to irrotational forces, and rigorous accuracy guarantees. Numerical experiments corroborate the theoretical findings and demonstrate the effectiveness of the approach.
\end{abstract}

\begin{keywords}
	Consistent Physics-informed neural networks; Oseen problem; collocation methods; optimal recovery; pressure-robust method. 
\end{keywords}

\begin{MSCcodes}
	41A25, 35Q35, 65N35, 68T07, 76D07
\end{MSCcodes}
\section{Introduction}
A fundamental model that arises in the study of viscous, incompressible flows at moderate or low Reynolds numbers is the stationary Oseen system, which can be seen as a linearization of the Navier–Stokes equations around a known convection field. 
Let  \(\Omega\subset \mathbb{R}^d, d=2,3\) be a bounded domain with Lipschitz boundary $\Gamma$. We aim to find a velocity field denoted as $\boldsymbol{u} : \Omega\to \mathbb{R}^d$ and the pressure field $p:\Omega\to \mathbb{R}$ from the following Oseen problem
\begin{align}\label{1}
	\begin{cases}
		- \nu \Delta \boldsymbol{u} + (\boldsymbol{\beta} \cdot \nabla) \boldsymbol{u} + \nabla p + \sigma \boldsymbol{u}&=\boldsymbol{f} \quad \text{in } \Omega, \\
		\hspace{3.6cm}\nabla \cdot \boldsymbol{u} &= 0\quad \text{in } \Omega, \\
		\hspace{4.1cm}\boldsymbol{u} &= \boldsymbol{g} \quad \text{on } \Gamma,
	\end{cases}
\end{align}
where the model involves the viscosity coefficient $\nu \in L^\infty(\Omega)$, the prescribed convection field $\boldsymbol{\beta} \in \left[W^{1,\infty}(\Omega)\right]^d$ and the reaction field $\sigma\in L^\infty(\Omega)$. To ensure the well‐posedness of the problem, for some positive constants $\nu_0, \nu_1$, and $\kappa$, these quantities satisfy the uniform bounds
\begin{align}
	0<\nu_0<\nu(x)<\nu_1,\quad \sigma(x)-\frac{1}{2}\nabla\cdot\boldsymbol{\beta}\geq \kappa.
\end{align}

It is well-known from \cite{badra2025oseen, girault1986finite} that, for the given assumptions on force data $ \boldsymbol{f}\in \left[{H}^{-1}(\Omega)\right]^d$ and boundary data $ \boldsymbol{g}\in \left[{H}^{1/2}(\Gamma)\right]^d$ such that
\begin{align}
	\int_{\Gamma}^{} \boldsymbol{g}\cdot \boldsymbol{n} \, ds = 0,
\end{align}
there exists a unique velocity field $\boldsymbol{u} \in \left[{H}^1(\Omega)\right]^d$ and pressure $p\in L^2_0(\Omega)$ which satisfies the model problem. Also, the solution will satisfy the following estimate 
\begin{align}\label{mainestimate}
	c\, \left(\| \boldsymbol{f}\|_{\boldsymbol{H}^{-1}(\Omega)}+\| \boldsymbol{g}\|_{ \boldsymbol{H}^{\frac{1}{2}}(\Gamma)} \right)\, \leq\| \boldsymbol{u}\|_{\boldsymbol{H}^1(\Omega)}+ \| p\|_{L^2(\Omega)}\leq C\, \left(\| \boldsymbol{f}\|_{\boldsymbol{H}^{-1}(\Omega)}+\| \boldsymbol{g}\|_{ \boldsymbol{H}^{\frac{1}{2}}(\Gamma)} \right),
\end{align}
for some constants $c,C>0$, depending only on $\Omega$.

The numerical approximation of incompressible flow equations, including~\eqref{1} has been widely studied using classical methods. Early foundational works of~\cite{girault1986finite} and later developments in~\cite{boffi2013mixed, DJ_silvester} establish the computational framework for stable velocity-pressure formulations. These advances were using classical discretizations which enforce the divergence constraint only weakly, which may lead to velocity errors that depend on the pressure. This issue has been addressed in the literature as the concept of pressure-robust numerical methods~\cite{linke2016pressure}. In particular, the works of Linke et al. in~\cite{ahmed2021pressure, john2017divergence, linke2019pressure, linke2018quasi} highlight that the velocity solution of incompressible flow problems depends only on the divergence-free component of the forcing term through the Helmholtz–Hodge decomposition, which improves the accuracy of velocity approximations in incompressible flow simulations. 

In recent years, PINNs introduced by Raissi et al.~\cite{karniadakis} have emerged as a mesh-free framework in which neural networks approximate PDE solutions while enforcing the governing equations and boundary conditions through the training loss. In application of this, PINNs have been successfully applied to a variety of problems, which include fluid dynamics and multiphysics systems~\cite{cai2021physics, cuomo2022scientific, karniadakis2021physics}. However, the theoretical understanding of PINNs remains an active area of research~\cite{gazoulis2023stability, MSMR, shin2023error, zeinhofer2024unified}. 
Recent work has therefore focused on developing rigorous theoretical foundations for neural network–based PDE solvers~\cite{lu, SYDK}. In particular, studies on the convergence and approximation properties of PINNs use tools from approximation theory and optimal recovery introduced by DeVore et al.~\cite{dahlkeoptimal, devore1989optimal, devore1993constructive, krieg2022recovery, novak2006function}. These developments have led to the concept of CPINNs in Besov regularity~\cite{BVPS,  khan2026mixedconsistentpinnselliptic, mishra2025priorierroranalysisconsistent, https://doi.org/10.1002/nme.70320}, where the loss functional is constructed to mimic the stability norms of the continuous problem, enabling provable approximation guarantees from pointwise data. 
Motivated by these developments, the present work combines ideas from pressure-robust numerical methods, consistent PINNs formulations, and optimal recovery theory to develop new CPINNs formulations for the stationary Oseen equations.

\subsection{Novel contribution}
\noindent\textbf{Main contributions.}
The main contributions of this work are as follows:
\begin{itemize}
\item We develop a CPINNs formulation for the stationary Oseen equations in which the loss functional is systematically derived from the stability structure of the continuous problem, providing a principled alternative to heuristic PINN designs.

\item We introduce a fundamentally new pressure-robust CPINNs formulation for the Oseen system. To the best of our knowledge, this is the first PINN framework that rigorously removes the influence of gradient forces from the velocity approximation.

\item We present a unified analysis covering both exactly divergence-free architectures and unconstrained velocity approximations, thereby bridging two distinct paradigms in physics-informed neural networks.

\item By leveraging optimal recovery theory, we establish, for the first time in the PINN/CPINN setting for Oseen-type problems, quantitative recovery estimates and error bounds for both velocity and pressure.

\item The proposed methodology constitutes a new class of pressure-robust CPINNs for incompressible flows, combining structural consistency, robustness with respect to irrotational forces, and provable accuracy guarantees.
\end{itemize}

These results establish a new analytical and computational framework for CPINNs applied to incompressible flow problems, significantly advancing the state of the art by integrating pressure-robustness with rigorous error control.

\subsection{Outline of the paper}
The remainder of the paper is organized as follows. In Section~\ref{Sec_2}, we introduce the functional framework and the CPINNs formulation for the Oseen equations. Section~\ref{Sec_3} develops the optimal recovery theory for the forcing term and boundary data, and also derives the corresponding recovery estimates for the velocity and pressure in the standard CPINNs formulation. In Section~\ref{Sec_4}, we study the divergence-free formulation and its associated recovery properties. Section~\ref{Sec_5} presents the pressure-robust CPINNs formulation and establishes optimal recovery rates for the velocity and pressure fields.  Finally, Section~\ref{Sec_6} reports numerical experiments illustrating the performance of the proposed methods.

\section{Function spaces and numerical framework}\label{Sec_2}
\subsection{Function spaces}
Let $W^{k,p}(\Omega)$ denote the usual Sobolev space of scalar-valued functions defined over a domain $\Omega$ for $k\geq 0$ is an integer and $1\leq p\leq \infty$. The associated norm is written as $\|\cdot\|_{W^{k,p}(\Omega)}.$ For the particular case $p=2$, we adopt the standard Hilbert space notation $W^{k,2}(\Omega) = H^{k}(\Omega)$. The dual space of $H^1_0(\Omega)$, denoted by $H^{-1}(\Omega)$ with the corresponding dual norm. For functions prescribed on the boundary, the trace space $H^{1/2}(\Gamma)$ is introduced as the range of the trace operator $\gamma: H^1(\Omega)\to L^2(\Gamma)$. All the boldface symbols will be used to denote the corresponding vector-valued Sobolev spaces, such as $\boldsymbol{H}^{k}(\Omega)= \left[{H}^{k}(\Omega)\right]^d$. The required function spaces for our analysis are introduced as follows:
\begin{align*}
	L^2_0(\Omega) &:= \{\, q \in L^2(\Omega) \;|\; \int_{\Omega} q\,dx = 0 \,\}, \\[4pt]
	\boldsymbol{H}^{\mathrm{div}}(\Omega) &:= \{\,\boldsymbol{v} \in \boldsymbol{L}^2(\Omega) \;|\; \nabla\!\cdot\!\boldsymbol{v} \in L^2(\Omega) \,\}, \\[4pt]
	\boldsymbol{H}^1_0(\Omega) &:= \{\, \boldsymbol{v} \in\boldsymbol{H}^1(\Omega) \;|\; \boldsymbol{v}|_{\Gamma} = 0 \,\}, \\[4pt]
	\boldsymbol{H}^{\mathrm{div}}_0(\Omega) &:= \{\, \boldsymbol{v} \in \boldsymbol{H}^{\mathrm{div}}(\Omega) \;|\; \boldsymbol{v}\!\cdot\!\boldsymbol{n} = 0 \text{ on } \Gamma \,\},  \\[4pt]
	\boldsymbol{V}_0(\Omega) &:= \{\, \boldsymbol{v} \in \boldsymbol{H}^{1}(\Omega) \;|\; \nabla \!\cdot\!\boldsymbol{v}= 0 \text{ in } \Omega \,\}.
\end{align*}

\textbf{Besov space:} Let \( 0 < s < \infty \) and \( 0 \leq p,q < \infty \). Then we introduce Besov spaces  \cite{BDS} on $\Omega\subset\mathbb{R}^d$ as 
\begin{align}
	\BB^{s}_{pq}(\Omega) = \left\{ f\in W^{[s]}_{p}(\Omega)\mid  |f|_{B^{s}_{pq}} =  \Bigg( \int_{\Omega} [t^{-s}\omega_{r}(f,t)_p]^q \,\frac{dt}{t} \Bigg)^{\frac{1}{q}} < \infty \right\} ,
\end{align}
where \(\omega_{r}(f,t)_{p}:= \sup_{|h|\leq t}\|\Delta^{r}_h f(x)\|_{L^p(\Omega_{h})}\),  \(t>0\), defined as the modulus of smoothness of $f$ and $\Delta^{r}_{h}$ represents the forward finite difference operator of order ${r>s}$ applied with a step size $h \in \mathbb{R}^d$ generated from \(\Omega_h:= \{x\in\Omega : [x,x+h]\subset\Omega\}\). The associated norm on the space is given as 
\begin{align}
	\|f\|_{B^s_{pq}(\Omega)} = \|f\|_{W^{[s]}_{p}(\Omega)}+ |f|_{B^{s}_{pq}(\Omega)}.
\end{align}
For the special case \( q' = \infty \), the Besov norm is computed by taking the supremum over \( \sup_{0 < t < h} \) instead of integrating with respect to \({dt}/{t} \). The corresponding function space is denoted by  $B^{s}_{p}(\Omega)$.
Based on the definition of the Besov space, we obtain
\begin{align}
	|f|^{}_{B^s_p(\Omega)}
	\asymp \sup_{k\geq 0} \omega_{r}(f,2^{-k})^{}_{p}\ 2^{ks}.
\end{align}
It follows that a function $f\in B^s_p(\Omega)$ if and only if it satisfies 
\begin{align}
	\omega_{r}(f,2^{-k})_{p} \leq 2^{-ks} 	|f|^{}_{B^s_p(\Omega)}, \quad k = 0,1,\ldots.
\end{align}
The relation $a \precsim b$ indicates the existence of a constant $C>0$, independent of $a$ and $b$, such that $a \le Cb$. Also, the notation $a \asymp b$ implies that there exist positive constants $C_1$ and $C_2$, independent of $a$ and $b$, such that $C_1b\leq a\leq C_2b$.

\subsection{Minimization problem}
We now proceed to define the theoretical loss functional based on the previously derived stability estimate \eqref{mainestimate}. This estimate allows us to establish an important equivalence relation between the error in the solution and the residuals of the governing PDE and associated boundary condition. Specifically, we have
\begin{align}\label{finalest1_rewritten}
	\|\boldsymbol{u} - \boldsymbol{v}\|_{\boldsymbol{H}^1(\Omega)} + \|p-q\|_{L^2(\Omega)} \asymp
	\|- \nu \Delta \boldsymbol{v} + (\boldsymbol{\beta} \cdot \nabla) \boldsymbol{v} + \nabla q + \sigma \boldsymbol{v}- \boldsymbol{f}\|_{\boldsymbol{H}^{-1}(\Omega)}\notag\\
	+ \|\nabla\cdot\boldsymbol{v}\|_{L^{2}(\Omega)} + \|\boldsymbol{g} - \boldsymbol{v}\|_{\boldsymbol{H}^{1/2}(\Gamma)}.
\end{align}
Motivated by the above equivalence, we introduce the theoretical loss functional $\L_T: V(=\boldsymbol{H}^1(\Omega)\times L_0^2(\Omega))\to \mathbb{R},$ defined as
\begin{align}\label{theoretical_loss}
	\mathcal{L}_T(\boldsymbol{v},q)
	= \|- \nu \Delta \boldsymbol{v} + (\boldsymbol{\beta} \cdot \nabla) \boldsymbol{v} + \nabla q + \sigma \boldsymbol{v}- \boldsymbol{f} \|_{\boldsymbol{H}^{-1}(\Omega)}^2 + \|\nabla\cdot\boldsymbol{v}\|^2_{L^{2}(\Omega)} \notag\\
	+ \|\boldsymbol g - \boldsymbol v\|_{\boldsymbol{H}^{1/2}(\Gamma)}^2.
\end{align}
By construction, the functional $\L_T$ is non-negative and strictly convex over $V$. Therefore, it admits a unique minimizer, which coincides with the exact solution of the governing PDE problem \eqref{1}.
Equivalently, the solution $(\boldsymbol{u}, p)$ can be characterized as
\begin{align}\label{1.6}
	(\boldsymbol{u}, p) = \mathop{\arg\min}_{(\boldsymbol{v}, q) \in  \boldsymbol{H}^1(\Omega)\times L_0^2(\Omega)} \mathcal{L}_T(\boldsymbol{v},q).
\end{align} 
However, the variational problem defined in~\eqref{1.6} involves minimization over the entire infinite-dimensional space $V$, which makes direct numerical treatment impractical. To make the problem computationally feasible, we restrict the minimization to a finite-dimensional approximation space $\N_n$.
\subsection{Numerical framework of collocation methods and PINNs}
In recent years, neural networks have become highly effective nonlinear function approximators for solving PDEs such as \eqref{1}. Let us suppose a neural network architecture, parameterized by $n$ trainable weights and biases, and $\N_n$ denotes a neural network approximation space. Within this framework, we seek an approximate solution $(\boldsymbol{\hat{u}}, \hat {p})\in \N_n$ obtained through the minimization of an appropriately defined loss functional in norm $\|\cdot\|_X$. This norm is constructed using sampled data points derived from the given functions $\boldsymbol{f}$ and $\boldsymbol{g}$ distributed across the whole domain $\Omega$ and on the boundary $\Gamma$.\\

\subsection*{Sampling of data} We define the collection of sampled data of the forcing term 
\begin{align}
	\boldsymbol{f} = (\boldsymbol{f}_1, \boldsymbol{f}_2, \ldots, \boldsymbol{f}_{\tilde{m}})\ \text{where}\ \boldsymbol{f}_i = [f^1(\boldsymbol{x}_i)\, f^2(\boldsymbol{x}_i)\ \cdots \ f^d(\boldsymbol{x}_i)]^T\ \text{for} \ i = 1,2\ldots \tilde{m},
\end{align}
corresponding to the set of sampling locations $\X: = \{\boldsymbol{x}_1, \boldsymbol{x}_2,\ldots \boldsymbol{x}_{\tilde m}\}$ over the domain. Similarly, the collection of sampled data for the boundary term is given by
\begin{align}
	\boldsymbol{g} = (\boldsymbol{g}_1, \boldsymbol{g}_2, \ldots, \boldsymbol{g}_{\bar{m}})\ \text{where}\ \boldsymbol{g}_j = [g^1(\boldsymbol{z}_j)\ g^2(\boldsymbol{z}_j)\ \cdots\ g^d(\boldsymbol{z}_j)]^T\ \text{for} \ j = 1,2\ldots \bar{m},
\end{align}
for the set of sampling locations $\Y: = \{\boldsymbol{z}_1, \boldsymbol{z}_2,\ldots \boldsymbol{z}_{\bar m}\}$ on the boundary.

In the PINNs approach, we seek an approximation $(\boldsymbol{\hat{u}}, \hat {p})\in \N_n$ that best fits the given data by minimizing the discrete loss functional. Then the general form of the minimization problem is
\begin{align}
	(\boldsymbol{\hat{u}}, \hat {p}) \in  \mathop{\arg\min}_{(\boldsymbol{v}, q) \in  \N_n} \mathcal{L}(\boldsymbol{v},q).
\end{align}
where the discrete least square loss is given by
\begin{align}
	\mathcal{L}(\boldsymbol{v},q)
	= \frac{1}{\tilde{m}}\sum_{l=1}^{d}\sum_{i=1}^{\tilde{m}}|- \nu \Delta v^l(\boldsymbol{x}_i)+ (\boldsymbol{\beta} \cdot \nabla) v^l(\boldsymbol{x}_i)+ \frac{\partial p}{\partial x^l}(\boldsymbol{x}_i) + \sigma v^l(\boldsymbol{x}_i)- f^l(\boldsymbol{x}_i)|^2\notag\\ +  \frac{1}{\tilde{m}} \sum_{i=1}^{\tilde{m}} |\nabla\cdot\boldsymbol{v} (\boldsymbol{x}_i)|^2
	+ \frac{1}{\bar{m}}\sum_{l=1}^{d}\sum_{j=1}^{\bar{m}}|g^l(\boldsymbol{z}_j) - v^l(\boldsymbol{z}_j) |^2.
\end{align}
for $\boldsymbol{x}_i = (x_i^1, x_i^2,\ldots, x_i^d)$. To achieve consistency with the underlying functional space $V$, a more suitable loss function $\L^*:V\to \mathbb{R}$ is proposed, which mimics the exact discrete version of continuous norms. Its squared form is given by 
\begin{align}
	\mathcal{L}^*(\boldsymbol{v},q)
	= \sum_{l=1}^{d}\hspace{-0.1cm} \Bigg[\frac{1}{\tilde{m}}\sum_{i=1}^{\tilde{m}}|- \nu \Delta v^l(\boldsymbol{x}_i)+ (\boldsymbol{\beta} \cdot \nabla) v^l(\boldsymbol{x}_i)+ \frac{\partial p}{\partial x^l}(\boldsymbol{x}_i) + \sigma v^l(\boldsymbol{x}_i)- f^l(\boldsymbol{x}_i)|^\gamma\Bigg]^{\frac{2}{\gamma}}\notag\\\hspace{-0.1cm}  +  \frac{1}{\tilde{m}} \sum_{i=1}^{\tilde{m}} |\nabla\cdot\boldsymbol{v}(\boldsymbol{x}_i) |^2
	+ \frac{1}{\bar{m}}\sum_{l=1}^{d}\sum_{j=1}^{\bar{m}}|g^l(\boldsymbol{z}_j) - v^l(\boldsymbol{z}_j)|^2 \notag\\+
	\frac{1}{\bar m^2}\sum_{l=1}^{d}
	\sum_{\substack{i,j=1 \\ i\neq j}}^{\bar m}
	\frac{\lvert [g - v](\boldsymbol{z}_i) - [g -v](\boldsymbol{z}_j) \rvert^2}{\lvert \boldsymbol{z}_i - \boldsymbol{z}_j\rvert^d}.
\end{align}
This modified formulation ensures that the loss is consistent with the true solution space $V$, leading to the consistent PINNs (CPINNs) formulation. The parameter $\gamma$ represents the smallest exponent for which the embedding $L^\gamma(\Omega)\hookrightarrow H^{-1}(\Omega)$ holds.

\section{Local polynomial approximation and interpolation}\label{Sec_3}
To approximate functions in Besov spaces, we employ local polynomial interpolation on a dyadic simplicial decomposition of $\Omega = (0,1)^d$. For a given parameter $r>1,$ we sample the function on the uniform tensor grid
\begin{align*}
	G_r := \left\{\left(\frac{j_1}{r-1}, \ldots,\frac{j_d}{r-1}\right), j_i = \{0,1,\ldots,r-1\}\right\}\subset[0,1]^d. 
\end{align*}
Each dyadic cube in the partition $\D_k$ is further subdivided into simplices $E\in \T_K,$ and on every simplex we use Lagrange interpolation associated with the polynomial space $\mathbb{P}^d_r$ of total degree less than $r$. The interpolant on $E$ is defined by $L_E(f)=\sum_{i}^{}f(x_i)\phi_{E,i}$, where $\phi_{E,i}$ is a basis function mapped from a reference simplex. This construction defines a uniformly bounded projection operator with Lebesgue constant $\Lambda_r$, leading to a near-best local approximation estimate
\begin{align}
	\|f-L_E(f)\|_{C(E)}\leq (1+\Lambda_r) \mathop{\inf}_{P \in  \mathbb{P}^d}\|f-P\|_{C(E)}.
\end{align} 
The interpolants constructed on each simplex collectively yield the global piecewise polynomial interpolant
\begin{align}\label{3.2}
	S_k^*(f):= \sum_{E\in \T_K}^{}L_E(f)\chi_E, 
\end{align}
which is continuous across element interfaces. 
\subsection*{Vector-valued interpolation in $\boldsymbol{H}^1(\Omega)$}

We now extend the previous scalar-valued interpolation results to vector-valued functions.  
Let $d \in \mathbb{N}$ and consider functions
\[
\boldsymbol{f} = (f_1,\dots,f_d) : \Omega \to \mathbb{R}^d,
\]
with components $f_i : \Omega \to \mathbb{R}$, $i=1,\dots,d$. We define the vector-valued Sobolev space
\[
\boldsymbol{H}^1(\Omega)
:= [H^1(\Omega)]^d
= \bigl\{ \boldsymbol{f} = (f_1,\dots,f_d) : f_i \in H^1(\Omega),\ i=1,\dots,d \bigr\},
\]
and equip it with the standard product norm
\[
\|\boldsymbol{f}\|_{\boldsymbol{H}^1(\Omega)}^2
:= \sum_{i=1}^d \|f_i\|_{H^1(\Omega)}^2.
\]
Similarly, for Besov spaces, we set
\[
\boldsymbol{B}^s_{pq}(\Omega)
:= \bigl[B^s_q(L^p(\Omega))\bigr]^d
= \bigl\{\boldsymbol{f} = (f_1,\dots,f_d) : f_i \in B^s_q(L^p(\Omega)),\ i=1,\dots,d \bigr\},
\]
with seminorm
\[
|\boldsymbol{f}|_{\boldsymbol{B}^s_{pq}(\Omega)}^2
:= \sum_{i=1}^d |f_i|_{B^s_{pq}(\Omega)}^2.
\]

Recall the scalar interpolation operator $L_E : C(T) \to \mathbb{P}_r$ on a simplex $E \subset I \in \mathcal{D}_k(\Omega)$ and the corresponding piecewise polynomial interpolant defined in~\eqref{3.2}.
For a vector-valued function $\boldsymbol{f} = (f_1,\dots,f_d)$ we define the vector-valued local
interpolation operator componentwise by
\[
\boldsymbol{L}_E(\boldsymbol{f})
:= \bigl(L_E(f_1),\dots,L_E(f_d)\bigr),
\]
and the global piecewise polynomial interpolant
\begin{align}\label{3.3}
	\boldsymbol{S}_k^*(\boldsymbol{f})
	:= \sum_{E \in \mathcal{T}_k} \boldsymbol{L}_E(\boldsymbol{f})\,\chi_E
	= \Bigl( S_k^*(f_1),\dots,S_k^*(f_d) \Bigr).
\end{align}
In particular, $\boldsymbol S_k^*(\boldsymbol{f})$ is continuous on $\Omega$, and each component belongs
to the same scalar finite-dimensional polynomial space as before. The following theorem characterizes the approximation accuracy of $\boldsymbol{S}_k^*$.
\begin{theorem}\label{thm_3.1}
	Assume $d\geq 2.$ Consider a vector field $\boldsymbol{f} = (f_1,f_2,\ldots,f_d)\in {\boldsymbol{B}^s_{pq}(\Omega)}$ with $s>d/p$ and $0<p\leq 2$. Let $\boldsymbol{S}_k^*$ be the vector-valued interpolation operator as defined in~\eqref{3.3}. Then for $k\geq 0$, $\boldsymbol{S}_k^*$ satisfies the estimate
	\begin{align}
		\|\boldsymbol{f} -\boldsymbol{S}_k^*(\boldsymbol{f})\|_{\boldsymbol{H}^1(\Omega)}\leq C\,|\boldsymbol{f} |_{\boldsymbol{B}^s_{pq}(\Omega)}\, 2^{-k\left(s-1-\frac{d}{p}+\frac{d}{2}\right)},
	\end{align}
	for a constant $C>0$ that does not depend on $\boldsymbol{f}$ and $k$.
\end{theorem}
\begin{proof}
	Since $\boldsymbol{B}^s_{pq}(\Omega)$ is defined componentwise, we can write
	\[
	\boldsymbol{f} = (f_1,\dots,f_d), 
	\qquad f_i \in B^s_{pq}(\Omega), \quad i=1,\dots,d,
	\]
	for the given assumptions $s > d/p$ and $0 < p \le 2$, so each component $f_i$ belongs to the scalar Besov space.
	Recall that $\boldsymbol{S}_k^*$ is defined by applying the scalar operator $S_k^*$ to each component of $\boldsymbol{f}$, that is,
	\[
	\boldsymbol{S}_k^*(\boldsymbol{f}) 
	:= \bigl(S_k^*(f_1),\dots,S_k^*(f_d)\bigr).
	\]
	For the scalar interpolation operator $S_k^*$, the result in~\cite{BVPS} states that there exists a constant $C>0$, independent of $k$ and of $f_i$, such that for every $k\ge 0$,
	\begin{equation}\label{eq:scalar-estimate}
		\|f_i - S_k^*(f_i)\|_{H^1(\Omega)}
		\;\le\;
		C\,|f_i|_{B^s_{pq}(\Omega)}\,
		2^{-k\left(s - 1 - \frac{d}{p} + \frac{d}{2}\right)},
		\qquad i=1,\dots,d.
	\end{equation}
	We endow $\boldsymbol{H}^1(\Omega)$ with the standard product norm defined as
	\[
	\|\boldsymbol{v}\|_{\boldsymbol{H}^1(\Omega)}^2
	:= \sum_{i=1}^d \|v_i\|_{H^1(\Omega)}^2,
	\qquad 
	\boldsymbol{v} = (v_1,\dots,v_d).
	\]
	Then, by definition of $\boldsymbol{S}_k^*$ and the above scalar estimate~\eqref{eq:scalar-estimate}, we obtain
	\begin{align*}
		\|\boldsymbol{f} - \boldsymbol{S}_k^*(\boldsymbol{f})\|_{\boldsymbol{H}^1(\Omega)}^2
		&= \sum_{i=1}^d \|f_i - S_k^*(f_i)\|_{H^1(\Omega)}^2
		\\
		&\le \sum_{i=1}^d 
		\Bigl(
		C\,|f_i|_{B^s_{pq}(\Omega)}\,
		2^{-k\left(s - 1 - \frac{d}{p} + \frac{d}{2}\right)}
		\Bigr)^2
		\\
		&= C^2\,2^{-2k\left(s - 1 - \frac{d}{p} + \frac{d}{2}\right)}
		\sum_{i=1}^d |f_i|_{B^s_{pq}(\Omega)}^2.
	\end{align*}
	Hence, after applying the square root, we obtain
	\begin{align*}
		\|\boldsymbol{f} - \boldsymbol{S}_k^*(\boldsymbol{f})\|_{\boldsymbol{H}^1(\Omega)}
		\;&\le\;
		C\,2^{-k\left(s - 1 - \frac{d}{p} + \frac{d}{2}\right)}
		\Biggl(\sum_{i=1}^d |f_i|_{B^s_{pq}(\Omega)}^2\Biggr)^{1/2}\\
		&= C\, |\boldsymbol{f}|_{\boldsymbol{B}^s_{pq}(\Omega)}\, 2^{-k\left(s - 1 - \frac{d}{p} + \frac{d}{2}\right)}.
	\end{align*}
\end{proof}
\subsection{Optimal recovery}
Optimal recovery theory provides a framework for quantifying an unknown function by reconstructing it from limited data. Our goal is to approximate the unknown solutions $\boldsymbol{u}$ and $p$ of \eqref{1} to accuracy $\varepsilon$ in the $\boldsymbol H^1(\Omega)$ and $L^2(\Omega)$ norm respectively, by using only finitely many samples of the data $\boldsymbol{f}$ and $\boldsymbol{g}$ observed at corresponding data sites $\X$ and $\Y$.   

We assume that the vector-valued source term and the boundary data defined through the trace operator satisfies
\begin{align}\label{eq:F-def}
	\begin{cases}
		\boldsymbol{f}\in \F:=U(\B), 
		\quad \B=\boldsymbol{B}^s_{pq}(\Omega), \quad 0<p,q\le\infty,\;\; s>d/p,\\
		\boldsymbol{g}\in \G:=\operatorname{Tr}(U(\bar \B)),
		\quad 
		\bar \B = \boldsymbol{B}^{\bar s}_{\bar p \bar q}(\Omega),\quad 
		\bar s > \frac{d}{\bar p},\;\; 0<\bar p\le 2,\,  \quad 0<\bar q\le\infty,
	\end{cases}
\end{align}
so that $\F$ embeds compactly into $\left[C(\Omega)\right]^d$ and $\G$ as a subset of $  \boldsymbol{H} ^{1/2}(\Gamma)$ compactly embeds into $\left[C(\Gamma)\right]^d$. 
These assumptions imply that any admissible solution belongs to
\begin{equation}\label{eq:U-def}
	\U := \{(\boldsymbol{u},p)\in \left[C(\Omega)\right]^{d+1}: (\boldsymbol{u},p) \text{ solves~\eqref{1} for } \boldsymbol{f}\in \F,\; \boldsymbol{g}\in \G\}.
\end{equation}

Given sampled data $\boldsymbol{f}=(\boldsymbol{f}_1,\dots,\boldsymbol{f}_m)$ and $\boldsymbol{g}=(\boldsymbol{g}_1,\dots,\boldsymbol{g}_{\bar m})$, the consistent data sets are
\begin{align}\label{eq:Fdata}
	\begin{cases}
		\F_{\text{data}} := \{\boldsymbol{f}\in \F: \boldsymbol{f}(\boldsymbol{x}_i)=\boldsymbol{f}_i,\; i=1,\dots,\tilde m\},\\
		\G_{\text{data}} := \{\boldsymbol{g}\in \G: \boldsymbol{g}(\boldsymbol{z}_j)=\boldsymbol{g}_j,\; j=1,\dots,\bar m\},
	\end{cases}
\end{align}
and the corresponding admissible solutions are
\begin{align*}
	\mathcal{U}_{\text{data}} &:= \left\{ (\boldsymbol{u},p)\in \U \, \middle| \, 
	\begin{aligned}
		- \nu \Delta \boldsymbol{u}(\boldsymbol{x}_i) + (\boldsymbol{\beta} \cdot \nabla) \boldsymbol{u}(\boldsymbol{x}_i) + \nabla p(\boldsymbol{x}_i) + \sigma \boldsymbol{u}(\boldsymbol{x}_i) &= \boldsymbol{f}_i,  \\
		\hspace{-2mm}\nabla\cdot \boldsymbol{u}(\boldsymbol{x}_i) = 0,\quad
		u(\boldsymbol{z}_j)&= \boldsymbol{g}_j.
	\end{aligned}
	\right\}
\end{align*}
We aim to determine how strongly the data space $\U_{\text{data}}$ determines the functions $\boldsymbol{u}$ and $p$.
To determine the information of $(\boldsymbol{f},\boldsymbol{g})$, recall that for a compact subset 
$K$ of a Banach space $X$, the Chebyshev ball $B(K)_X$ with radius 
\begin{equation}\label{eq:rad}
	R(K)_X := \operatorname{rad}(B(K))_X,
\end{equation}
describes the optimal recovery error. Thus, the OR error for solution $(\boldsymbol{u},p)$ is given by $R(\U_{\text{data}})_X = \operatorname{rad}(B(\U_{\text{data}}))_X.$ For fixed data sites $(\X,\Y)$, the uniform OR rate over all admissible data is defined as
\begin{equation}\label{eq:uniform-OR}
	R^*(\U,\X,\Y)_X := \sup_{(\boldsymbol{u},p)\in \U} R(\U_{\text{data}}(\boldsymbol{u},p))_X.
\end{equation}
If $m = |\X|+|\Y|$ is the total sampling budget, then the uniform OR rate is given as
\begin{equation}\label{eq:RmU}
	R^*_m(\U)_X := 
	\inf_{\substack{\X\subset\Omega,\; \Y\subset\Gamma\\ |\X|+|\Y|=m}}
	R^*(\U,\X,\Y)_X.
\end{equation}
Analogous definitions give the OR rates $R^*(\F,\X)_X$, $R^*_{\tilde{m}}(\F)_X$, and the corresponding quantities for recovering $\boldsymbol{g}$ from $\G$.
\subsubsection{Optimal recovery of \texorpdfstring{$\boldsymbol{f}$}{f}}
Let $\boldsymbol{f}:\Omega\to \mathbb{R}^d$ be a vector-valued function belong to the unit ball of Besov space
\begin{align}
	\F = U(\boldsymbol{B}^s_{pq}(\Omega)), \quad s>d/p, \ 0<p,q\leq \infty.
\end{align}
We will show that the optimal decay rate for recovering the function $\boldsymbol{f}$ from $\tilde m$ sampling points is given by
\begin{align*}
	R^*_{\tilde{m}}(\F,\X)\asymp\tilde{m}^{-\alpha_X}.
\end{align*}
\begin{theorem}\label{thm4.1}
	Consider the domain $\Omega = (0,1)^d$ and a unit ball of Besov space $U(\boldsymbol{B}^s_{pq}(\Omega))$ for $s>d/p$ and $0<p,q\leq \infty$. For the Banach space $X= \boldsymbol{H}^{-1}(\Omega)$, the following OR rate holds:
	\begin{itemize}
		\item If either $d\geq 3$, or $d=2$ with $p>1,$ then for $m\geq 1$
		\begin{align}
			R^*_{\tilde m}(\F)\asymp\tilde{m}^{-\alpha_{-1}},
		\end{align}
		\item When $d=2$ with $0<p\leq 1$, the rate will become
		\begin{align}\label{4.9}
			\tilde{m}^{-\alpha_{-1}}\lesssim R^*_{\tilde m}(\F)\lesssim\tilde{m}^{-\alpha_{-1}}\log(\tilde{m}),
		\end{align}
	\end{itemize}
	where the constants appear in equivalence are free from $\tilde{m}$. The exponent is given by $\alpha_{-1} = \frac{s}{d}-\left[\frac{1}{p}-\frac{1}{\gamma}\right]_+$ for $ \frac{1}{\gamma} = \frac{1}{2}+\frac{1}{d}$.
\end{theorem}
\begin{proof}
	The proof is provided in the Appendix~\ref{App_A}.
\end{proof}
\subsubsection{Optimal recovery of \texorpdfstring{$\boldsymbol{g}$}{g}}
In this section, we study optimal recovery of boundary traces in $\boldsymbol{H}^{1/2}(\Gamma)$. Let us define $\overline{\B} = \boldsymbol{B}^{\bar s}_{\bar{p}\bar{q}}(\Omega)$ and the model class $\G := \{\boldsymbol{g}=\operatorname {Tr}(\boldsymbol{v})\mid \|\boldsymbol{v}\|_{\overline{\B}}\leq 1\}$. Assume the smoothness condition $\bar s>d/\bar p$ and $0<\bar p\leq 2$, then the functions in $U(\B)$ are continuous and compactly embedded in $\boldsymbol{H}^1(\Omega)$. Let $G_{k,r}$ be the tensor-product grid over $\overline{\Omega} = [0,1]^d$. Therefore the sampling on the boundary uses the points
\begin{align}
	\Y := G_{k,r}\cap\Gamma = \{\boldsymbol z_i\}_{i=1}^{\bar m},\ \text{for} \ \bar{m}\asymp 2^{k(d-1)}. 
\end{align}
Let $\overline{ \phi}_i$ are the trace of the Lagrange basis functions $ \phi_i$ over the boundary. For $\boldsymbol  g\in\G$, define the interpolation polynomial on the boundary as
\begin{align}
	\overline{\boldsymbol S}_k(\boldsymbol g):= (\overline{S}_k(g_1), \ldots, \overline{S}_k(g_d)) \ \text{where } \overline{S}_k(g_i)= \sum_{j=1}^{\bar m} g_i(\boldsymbol z_j)\overline{\phi}_j, \quad i= 1,\ldots,d.
\end{align}
If $\boldsymbol v\in \B$ be any function with trace $\boldsymbol g$ and $\boldsymbol S^*_k(v)$ is its Lagrange interpolant on $G_{k,r}$, then $\operatorname{Tr}(\boldsymbol S^*(\boldsymbol v)) = \overline{\boldsymbol S}_k(\boldsymbol g)$, because interior basis functions vanish on the boundary. Since $\boldsymbol S^*(\boldsymbol v)\in \boldsymbol H^1(\Omega)$, this implies $\overline{\boldsymbol S}_k(\boldsymbol g)\in \boldsymbol H^{1/2}(\Gamma)$.
\begin{theorem}\label{thm4.2}
	Consider the Besov ball $\overline{\B} = \boldsymbol  B^{\bar{s}}_{\bar{p}\bar{q}}$ with smoothness conditions $\bar{s}>d/\bar{p}$, $0<\bar p\leq 2$ and $\bar{q}\in (0,\infty]$. Define the associated model class $\G = \{Tr(\boldsymbol v): \|\boldsymbol v\|_{\overline{\B}}\leq 1\}$. Then, for $\bar m$ the optimal recovery error of functions in $\G$ under the norm $\boldsymbol H^{1/2}(\Gamma)$ given by
	\begin{align}\label{4.24}
		R^*_{\bar{m}}(\G)\asymp \bar{m}^{-\beta}, \quad \text{for}\quad \beta = \frac{\bar s -1}{d-1}-\frac{d}{d-1}\left(\frac{1}{\bar p}-\frac{1}{2}\right),
	\end{align} 
	with the constants appeared in equivalence are free from $\bar{m}$. 
\end{theorem}
\begin{proof}
	The complete proof can be found in the Appendix~\ref{App_B}.
\end{proof}
\subsubsection{Optimal recovery of \texorpdfstring{$\boldsymbol{u}$}{u} and $p$}
In this section we study optimal recovery of solution $\boldsymbol u$ and pressure $p$ of the boundary value problem
\begin{align}
	\begin{cases}
		- \nu \Delta \boldsymbol{u}(\boldsymbol{x}_i) + (\boldsymbol{\beta} \cdot \nabla) \boldsymbol{u}(\boldsymbol{x}_i) + \nabla p(\boldsymbol{x}_i) + \sigma \boldsymbol{u}(\boldsymbol{x}_i)&=\boldsymbol{f}_i \quad \text{in } \Omega, \\
		\hspace{3.6cm}\nabla \cdot \boldsymbol{u}(\boldsymbol{x}_i) &= 0\quad \text{in } \Omega, \\
		\hspace{4.1cm}\boldsymbol{u}(\boldsymbol{z}_i) &= \boldsymbol{g}_i \quad \text{on } \Gamma,
	\end{cases}
\end{align}
where the interior data $\{\boldsymbol f(\boldsymbol x_i)\}$ and boundary data $\{\boldsymbol g(\boldsymbol z_j)\}$ are observed at $\X$ and $\Y$ correspondingly for total budget $m=\tilde{m}+\bar{m}$. For the model classes $\F = U(\boldsymbol B^s_{pq}(\Omega))$ and $\G = \operatorname{Tr}(U(\boldsymbol B^{\bar s}_{\bar p \bar q}))$ with $s>d/p,\, p,q,\bar q\in (0,\infty]\, \bar s>d/\bar p$ and  $\bar p\in(0,2]$, the solution space is given by
\begin{align*}
	\U := \{( \boldsymbol{u},p)\in C(\overline{\Omega})^{d+1}\mid \text{$( \boldsymbol{u},p)$ solves \eqref{1} for } f\in\F, g\in\G\}.
\end{align*}
\begin{theorem}
	For $\Omega = (0,1)^d$ and $\F, \G$ and $\U$ as defined above. Then 
	\begin{itemize}
		\item If $d\geq3$ or $d=2$ with $p>1$, the following optimal recovery rates hold
		\begin{align}
			R^*_{m}(\U)_{X} \asymp m^{-\min\{\alpha,\beta\}},  \quad m\geq 2,
		\end{align}
		\item If $d=2$ and $p\in (0,1]$, the following optimal recovery rate hold
		\begin{align}
			m^{-\min\{\alpha,\beta\}} \lesssim R^*_m(\U)_{X}\lesssim \log(m) m^{-\alpha} + m^{-\beta},\quad m\geq 2,
		\end{align}
	\end{itemize}
	under the equivalence constants independent of $m$ and $X := \boldsymbol{H}^1(\Omega)\times L^2(\Omega)$.
\end{theorem}
\begin{proof}
	Let two admissible solutions $(\boldsymbol{u}_1,p_1), (\boldsymbol{u}_2,p_2)\in\U$ corresponds to the data $(\boldsymbol{f}_1,\boldsymbol{g}_1)$ and $(\boldsymbol{f}_2,\boldsymbol{g}_2)$. The stability estimate \eqref{mainestimate} yields
	\begin{align}
		\|\boldsymbol{u}_1-\boldsymbol{u}_2\|_{\boldsymbol{H}^1(\Omega)} + \|p_1-p_2\|_{L^2(\Omega)} \asymp \|\boldsymbol{f}_1-\boldsymbol{f}_2\|_{\boldsymbol{H}^{-1}(\Omega)} + \|\boldsymbol{g}_1-\boldsymbol{g}_2\|_{\boldsymbol{H}^{1/2}(\Gamma)}.
	\end{align}
	Hence for fixed sampling sets,
	\begin{align}
		R^*(\U,\X,\Y)_{X} \asymp R^*(\F,\X)_{\boldsymbol{H}^{-1}(\Omega)} + R^*(\G,\Y)_{\boldsymbol{H}^{1/2} (\Gamma)}.
	\end{align}
	After splitting the budget evenly $\tilde m = \bar m = m/2$ and applying the recovery rate for $\F$ and $\G$ from Theorem~\ref{thm4.1} and Theorem~\ref{thm4.2}, we get
	\begin{align*}
		R^*_m(\U) \lesssim m^{-\alpha} + m^{-\beta}\lesssim m^{-\min\{\alpha,\beta\}} \quad \text{for } \ d\geq 3,
	\end{align*}
	and for $d=2$ and $p\leq 1$, we get recovery rate $R^*_m(\U)\lesssim \log (m)m^{-\alpha} + m^{-\beta}$. Conversely, the right-hand side of same theorems for any allocation $m = \tilde m + \bar m$ gives a lower bound
	\begin{align*}
		R^*_m(\U)&\gtrsim \inf_{m = |\X| + |\Y|} \left( R^*(\F,\X) + R^*(\G,\Y)\right)\\
		&\gtrsim \inf_{m=\tilde m+\bar m} \left(\tilde{m}^{-\alpha} + \bar{m}^{-\beta}\right)\gtrsim m^{-\min\{\alpha,\beta\}}.
	\end{align*}
	The proof is completed by combining the upper and lower bounds.
\end{proof}
\subsubsection{Final optimal recovery summary}\label{sec_4.4}
Optimal recovery analysis shows that many Besov model classes for $\boldsymbol{f},\boldsymbol{g}, \boldsymbol{u}$ and $p$ yield the same optimal recovery rates. Thus, only the largest classes attaining these rates are of primary relevance.
\subsection*{Largest model classes for $\F$}
For $d\geq 3$ or $d=2$ with $p>1$, all classes corresponding to $\boldsymbol{f}$ that achieve the recovery rate $m^{-s/d}$ in $\boldsymbol{H}^{-1}(\Omega)$ are contained in $\F = U(\boldsymbol{B}^s_{p,\infty}(\Omega))$, with $p\geq \gamma$ and $s>d/p$. In case of $d=2, p\in (0,1]$, the largest model class $\F = U(\boldsymbol{B}^s_{1,\infty}(\Omega))$ with $s>2$ satisfies
\begin{align*}
	m^{-s/2}\lesssim R^*_m(\F)_{\boldsymbol{H}^{-1}(\Omega)} \lesssim (1+\log(m))m^{-s/2}.
\end{align*}

\subsection*{Largest model classes for $\G$} 
All the Besov model classes with recovery rate $m^{-(s-1)/(d-1)}$ in $H^{1/2}(\Gamma)$ lie inside $\G = \operatorname{Tr} (U(\boldsymbol{B}^{\bar s}_{2,\infty}(\Omega)))$ with $\bar s>d/2$, and satisfies
\begin{align*}
	R^*_m(\G)_{\boldsymbol{H}^{1/2}(\Gamma)} \asymp m^{-(\bar s-1)/(d-1)}.
\end{align*}

\subsection*{Model class for solutions \texorpdfstring{$\boldsymbol{u}$}{u} and $p$}
Using the above maximal model classes for $\F$ and $\G$, the PDE solution $(\boldsymbol{u},p)$ with $\tilde{m}=\bar{m} = m/2$ satisfies
\begin{align*}
	R^*_m(\U)_X \asymp m^{-\min\{s/d, (\bar{s}-1)/(d-1)\}}.
\end{align*}

\subsection{Discrete residual control}
In order to make theoretical loss functional $\L_T$ computationally feasible, we introduce corresponding discrete loss functional $\L^*$ based on sampled data of $\boldsymbol f$ and $\boldsymbol g$. Under the model class assumptions $\boldsymbol f\in\F= U(\B)$, $\boldsymbol g\in\G$, the exact solution $(\boldsymbol u,p$) belongs to a compact solution class $\U$ with optimal recovery rate is  $\{\tilde{m}^{-s/d}, \bar{m}^{-(s-1)/(d-1)}\}$. For any $(\boldsymbol v,q)$, we define the model norm as $\|(\boldsymbol v,q)\|_{\U} =\max\{\|- \nu \Delta \boldsymbol{v} + (\boldsymbol{\beta} \cdot \nabla) \boldsymbol{v} + \sigma \boldsymbol{v} + \nabla q\|_{\B}, \|\operatorname{Tr}(\boldsymbol v)\|_{\operatorname{Tr}(\overline{\B})}\}$. The following result establishes control of the approximation error by the discrete residual.

\begin{theorem}\label{thm_3.5}
	Let $(\boldsymbol u,p)$ be the solution of the Oseen problem~\eqref{1}. Under the above assumptions over $\boldsymbol f, \boldsymbol g$ for $d=2,3,$ and for any $(\boldsymbol v,q)\in \boldsymbol{H}^{1}(\Omega) \times \boldsymbol{L}^{2}_0(\Omega)$, the following estimate holds 
	\begin{align}
		\|\boldsymbol u-\boldsymbol v\|^2_{\boldsymbol H^1(\Omega)} + \|p-q\|^2_{L^2(\Omega)} \lesssim \L^*(v,q) + (\|(\boldsymbol v,q)\|^2_{\U}+1)\mathfrak{R}_U(\tilde m, \bar m),
	\end{align}
	where the implicit constant is independent of $\boldsymbol u, \boldsymbol v, \tilde m, \bar m$ and the residual term is given by
	\[\mathfrak{R}_U(\tilde m, \bar m) = \max\{\tilde{m}^{-2s/d},\bar{m}^{-2(\bar s-1)/(d-1)}\}.\]
\end{theorem}
\begin{proof}
	Using the stability estimate~\eqref{finalest1_rewritten}, we get
	\begin{align}
		\|\boldsymbol u-\boldsymbol v\|^2_{\boldsymbol H^1(\Omega)} + \|p-q\|^2_{L^2(\Omega)} \lesssim \|R(\boldsymbol v,q)\|^2_{\boldsymbol H^{-1}(\Omega)} + \|\nabla\cdot \boldsymbol v\|^2_{L^2(\Omega)} + \|\boldsymbol g-\boldsymbol v\|^2_{\boldsymbol H^{1/2}(\Gamma)},
	\end{align}
	where $R(\boldsymbol v,q): = - \nu \Delta \boldsymbol{v} + (\boldsymbol{\beta} \cdot \nabla) \boldsymbol{v} +\nabla q +  \sigma \boldsymbol{v} -\boldsymbol{f}$. By invoking the norm equivalence between continuous and discrete norms for $H^{-1}(\Omega), L^2(\Omega)$ and $H^{1/2}(\Gamma)$-norm (see~\cite[Lemma~6.1 and Theorem~6.4]{BVPS}), together with the Sobolev embedding $L^\gamma(\Omega)\hookrightarrow H^{-1}(\Omega)$ for $1/\gamma = 1/d +1/2$, it follows
	\begin{align*}
		\|\boldsymbol u-\boldsymbol v\|^2_{\boldsymbol H^1(\Omega)} &+ \|p-q\|^2_{L^2(\Omega)} \lesssim \|R(\boldsymbol v,q)\|^2_{\boldsymbol H^{-1}(\Omega)} + \|\nabla\cdot \boldsymbol v\|^2_{L^2(\Omega)} + \|\boldsymbol g-\boldsymbol v\|^2_{\boldsymbol H^{1/2}(\Gamma)}\\
		&\lesssim \|R(\boldsymbol v,q)\|^{2,*}_{\boldsymbol L^{\gamma}(\Omega)} + \|\nabla\cdot \boldsymbol v\|^{2,*}_{L^2(\Omega)} + \|\boldsymbol g-\boldsymbol v\|^{2,*}_{\boldsymbol H^{1/2}(\Gamma)}\\		
		&  + \|R(\boldsymbol v,q)\|^{2}_{\B}\tilde{m}^{-\frac{2s}{d}} + \|\nabla\cdot \boldsymbol v\|^{2,*}_{\B}\tilde{m}^{-\frac{2s}{d}}  + \|\boldsymbol g-\boldsymbol v\|^{2,*}_{\operatorname{Tr}(\overline{\B})}\bar{m}^{-\frac{2\bar s-1)}{(d-1)}}\\
		&\lesssim	\L^*(v,q) + (\|(\boldsymbol v,q)\|^2_{\U}+1)\mathfrak{R}_U(\tilde m, \bar m).
	\end{align*}
	In the last step, we conclude the proof by using the bounds for norms of the data $f\in U(\B)$ and $g\in \operatorname{Tr}(U(\overline{\B}))$.
\end{proof}

\section{Model problem with divergence-free constraint}\label{Sec_4}
In this section, we restrict our attention to the two-dimensional case, i.e., $d=2$. In many applications involving incompressible flows, it is essential to enforce the divergence-free constraint on the velocity field exactly, rather than only in a weak or penalized sense. In this section, we present a self-contained formulation of the divergence-free model problem, introduce the associated loss functional, and derive the corresponding optimal recovery results. 

\subsection{Functional setting for the divergence-free Oseen problem}
Let $\Omega=(0,1)^2$ and consider the boundary value problem~\eqref{1}, where the interior observations $\{\boldsymbol f(\boldsymbol x_i)\}_{i=1}^{\tilde m}$ and boundary observations $\{\boldsymbol g(\boldsymbol z_j)\}_{j=1}^{\bar m}$ are sampled at data sites $\X$ and $\Y$, with total sampling budget $m=\tilde m+\bar m$. As before, we assume the model classes
\[
{\F} = U( \boldsymbol{B}^s_{pq}(\Omega)), \qquad \G = \operatorname{Tr}(U( \boldsymbol{B}^{\bar s}_{\bar p\bar q}(\Omega))),
\]
with $s>d/p$, $\bar{s}>d/\bar{p}$, $p,q,\bar q \in (0,\infty]$, $\bar p \in (0,2]$, so that $ \boldsymbol{f} \in {\F}$ and $\boldsymbol{g}\in \G$ are continuous and the PDE \eqref{1} admits a unique solution $(\boldsymbol{u}, p)\in \boldsymbol{V}_0(\Omega) \times L^2_0(\Omega)$ satisfying the standard stability estimate
\begin{align}
	\|\boldsymbol{u}_1-\boldsymbol{u}_2\|_{\boldsymbol{H}^1(\Omega)} + \|p_1-p_2\|_{L^2(\Omega)}
	\asymp \|\boldsymbol{f}_1-\boldsymbol{f}_2\|_{\boldsymbol{H}^{-1}(\Omega)} + \|\boldsymbol{g}_1-\boldsymbol{g}_2\|_{\boldsymbol{H}^{1/2}(\Gamma)} .
	\label{eq:DF-stability}
\end{align}
Thus, we define the corresponding solution class
\[
\U_{\mathrm{div}} := \{(\boldsymbol{u},p)\in C(\overline{\Omega})^{d+1} : (\boldsymbol{u},p)\ \text{solves \eqref{1} for some } \boldsymbol{f}\in{\F},\, \boldsymbol{g}\in \G\}.
\]

\subsection{Loss functional for divergence-free formulation}
Motivated by the stability estimate, we introduce a theoretical loss functional that measures the residual of the momentum equation in $\boldsymbol{H}^{-1}(\Omega)$ together with the boundary data. For admissible $(v,q)\in \boldsymbol{V}_0(\Omega)\times L^2_0(\Omega),$ we define theoretical loss functional
\begin{align}
	\mathcal{L}_{\mathrm{div}}(\boldsymbol{v},q)
	&=
	\bigl\| -\nu\Delta\boldsymbol{v} + (\boldsymbol{\beta}\cdot\nabla)\boldsymbol{v} + \nabla q+  \sigma \boldsymbol{v} - \boldsymbol{f} \bigr\|^2_{\boldsymbol{H}^{-1}(\Omega)}
	+
	\|\boldsymbol{g}-\boldsymbol{v}\|^2_{\boldsymbol{H}^{1/2}(\Gamma)}. 
	\label{eq:DF-loss}
\end{align}
The functional $\mathcal{L}_{\mathrm{div}}$ is non-negative and convex on $\boldsymbol{V}_0(\Omega)\times L^2_0(\Omega)$, so it has a unique minimizer, which is exactly the solution of the divergence-free problem. Hence,
\begin{align*}
	(\boldsymbol{u},p)
	=
	\mathop{\arg\min}_{(\boldsymbol{v},q)\in \boldsymbol{V}_0(\Omega)\times L_0^2(\Omega)}
	\mathcal{L}_{\mathrm{div}}(\boldsymbol{v},q).
\end{align*}
The minimizer of $\mathcal{L}_{\mathrm{div}}$ over $\boldsymbol{V}_0(\Omega)\times L^2_0(\Omega)$ coincides with the exact divergence–free velocity solution and pressure of the Oseen problem. The discrete version follows from replacing the norms by their sampled analogues at data sites $(\X,\Y)$, which leads to divergence-free CPINNs formulation.

\subsection{Optimal recovery rate for the divergence-free formulation}
In the divergence free formulation, the divergence term does not appear explicitly in loss, 
and the velocity field is determined with pressure through the forcing term $\boldsymbol f$ 
and the boundary data $\boldsymbol g$.  Hence the recovery accuracy of $(\boldsymbol u, p)$ depends only on the information content of the data classes $\F$ and $\G$, as stated in the theorem below.
\begin{theorem}\label{thm5.2}
	For the divergence-free formulation, model problem \eqref{1} is considered with model classes $\F$ and $\G$ as above. The optimal recovery rate of the solution class satisfies
	\[
	R^*_m(\U_{\mathrm{div}})_{H^1(\Omega)}
	\asymp m^{-\min\{\alpha,\beta\}}, \qquad m\ge 2,
	\]
	for $d=2$ with $p>1$.  
	When $d=2$ and $p\in(0,1]$,
	\[
	m^{-\min\{\alpha,\beta\}}
	\lesssim R^*_m(\U_{\mathrm{div}})
	\lesssim \log(m)\,m^{-\alpha} + m^{-\beta}.
	\]
	All equivalence constants are independent of $m$.
\end{theorem}

\begin{proof}
	Let the total sampling budget be $m=\tilde m+\bar m$.  
	Using the stability estimate \eqref{eq:DF-stability} and the optimal recovery rates established for $\F$ and $\G$ in Theorems~\ref{thm4.1}-\ref{thm4.2}, we obtain
	\[
	R^*(\U_{\mathrm{div}},\X,\Y)_{ \boldsymbol{H}^1(\Omega)}
	\asymp
	R^*(\F,\X)_{ \boldsymbol{H}^{-1}(\Omega)} + R^*(\G,\Y)_{ \boldsymbol{H}^{1/2}(\Gamma)} .
	\]
	Splitting the budget evenly, $\tilde m = \bar m = m/2$, yields the upper bound
	\[
	R^*_m(\U_{\mathrm{div}}) \lesssim m^{-\alpha} + m^{-\beta}
	\lesssim m^{-\min\{\alpha,\beta\}}
	\quad \text{for } p>1.
	\]
	If $p\le 1$, we obtain
	\[
	R^*_m(\U_{\mathrm{div}})
	\lesssim \log(m)\,m^{-\alpha} + m^{-\beta}.
	\]
	For the lower bound, optimizing over all allocations $\tilde m+\bar m = m$ gives
	\[
	R^*_m(\U_{\mathrm{div}})\gtrsim \inf_{m = |\X| + |\Y|} \left( R^*(\F,\X) + R^*(\G,\Y)\right)\
	\gtrsim m^{-\min\{\alpha,\beta\}}.
	\]
\end{proof}

\subsection{Largest model classes for recovery in the divergence-free formulation}

Since the velocity and pressure terms appears in a solution class $\U_{div}$, the optimal rate is dictated by the largest model classes for $\F$ and $\G$ as
\[
R^*_m(\U_{\mathrm{div}})_{H^1(\Omega)}
\asymp m^{-\min\{s/d,\;(\bar{s}-1)/(d-1)\}}.
\]
Thus, the same maximal Besov classes identified in Section~\ref{sec_4.4} yield the best possible recovery rate for the divergence-free velocity and pressure field.

\begin{theorem}[Discrete residual control for divergence-free formulation]
	Let $\boldsymbol u$ be the solution of the divergence-free Oseen problem. 
	Under the assumptions on the data $\boldsymbol f\in U(\B), \boldsymbol g\in \operatorname{Tr}(U(\overline{\B}))$ for $d=2$, and for any $(\boldsymbol v,q )\in \boldsymbol V_0(\Omega)\times L^2_0(\Omega)$ the following estimate holds 
	\begin{align}
		\|\boldsymbol u-\boldsymbol v\|^2_{\boldsymbol H^1(\Omega)} + \|p-q\|^2_{L^2(\Omega)}
		\lesssim \L^*_{\mathrm{div}}(\boldsymbol v, q) 
		+ (\|(\boldsymbol v,q)\|^2_{\U}+1)\mathfrak{R}_U(\tilde m, \bar m),
	\end{align}
	where the constant is independent of $\boldsymbol u, \boldsymbol v, \tilde m, \bar m$.
\end{theorem}
\begin{proof}
	The result follows by repeating the argument of Theorem~\ref{thm_3.5}, using the corresponding norm equivalences and optimal recovery bounds.
\end{proof}

The methods proposed in Section~\ref{Sec_3} for the Oseen problem are not pressure robust. The formulation developed in Section~\ref{Sec_4} improves the divergence-free structure but is still not fully pressure robust, since the velocity may depend on gradient forces. Therefore, in the next section, we propose a fully pressure-robust formulation for the Oseen problem, where the velocity approximation error does not depend on the pressure.
\section{Pressure-robust CPINNs formulation for the Oseen equation}\label{Sec_5}
In this section, we restrict our attention to the two-dimensional setting, i.e., $d=2$. 
\subsection{Velocity equivalence and pressure robustness}
A fundamental observation in the numerical analysis of the incompressible Oseen problem is that the velocity solution depends only on the equivalence class of the forcing term modulo gradients. Two forces $\boldsymbol f_1, \boldsymbol f_2\in [H^{-1}(\Omega)]^2$ are said to be velocity-equivalent if they lead to the same velocity solution, which holds if and only if
\begin{align*}
	\boldsymbol f_2 = \boldsymbol f_1+\nabla \phi \text{ for some } \phi \in H^{1}(\Omega).
\end{align*}
Thus, a numerical method is considered pressure-robust if the error in the velocity approximation is independent of the pressure, which means that large or rapidly varying pressure gradients do not affect the velocity error. Equivalently, the velocity approximation depends only on the divergence-free component of the forcing term.
In particular, as shown in~\cite{linke2016pressure} for the Stokes problem and extending naturally to the Oseen equations, the exact velocity solution is determined solely by the Helmholtz–Hodge projector of the forcing term, while the irrotational (gradient) component influences only the pressure and has no effect on the velocity field.
%

\begin{definition}[Helmholtz–Hodge decomposition]
	Let $\Omega \subset \mathbb{R}^2$ be a bounded domain. Then every vector field $\boldsymbol{f} \in \boldsymbol L^2(\Omega)$ admits a unique decomposition of the form (see \cite{linke2019pressure})
	\begin{align*}
		\boldsymbol{f} = \nabla \phi + \mathbb{P}(\boldsymbol{f})
	\end{align*}
	where $\phi \in H^1(\Omega)/\mathbb{R}$, and 
	\begin{align*}
		\mathbb{P}(\boldsymbol{f})\in \mathbb{L}^2_{\sigma}(\Omega):= \{\boldsymbol v\in \boldsymbol L^2(\Omega): \nabla\cdot \boldsymbol v=0 \, \text{in}\,  \Omega\}
	\end{align*}
	is the Helmholtz–Hodge projector of $\boldsymbol{f}$. Moreover, the gradient component $\nabla \phi$ and the divergence-free component $	\mathbb{P}(\boldsymbol{f})$ are orthogonal in $\boldsymbol{L}^2(\Omega)$.
	
	The Helmholtz–Hodge projection admits a natural extension beyond $\boldsymbol{L}^2(\Omega)$. In particular, it can be continuously extended on the dual space $\boldsymbol{H}^{-1}(\Omega),$ with values in $ \boldsymbol{V}^*_0(\Omega)$, the dual of the divergence-free space $\boldsymbol{V}_0(\Omega)$.
\end{definition}


\subsection{Pressure-robust theoretical and discrete loss functional}
In order to formulate the pressure-robust Oseen formulation, we define the following theoretical loss functional
\begin{align}\label{6.3}
	\L_{\PR}(\boldsymbol{v}) = \|\nabla \times \left(-\nu\Delta\boldsymbol{v} +  (\boldsymbol{\beta}\cdot\nabla) \boldsymbol{v} + \sigma \boldsymbol{v}-\boldsymbol{f} \right) \|^2_{\boldsymbol{H}^{-2}(\Omega)} + \|\boldsymbol{g}-\boldsymbol{v}\|^2_{\boldsymbol{H}^{1/2}(\Gamma)} .
\end{align}
This loss penalizes only the curl of the momentum residual and therefore annihilates all gradient forces. As a result, the pressure variable does not influence the velocity approximation. Consequently, minimizing $	\L_{\PR}$
yields a velocity approximation that depends only on the divergence-free component of the forcing term.

Let $\N_n$ denotes a neural network approximation space. Given interior sampling points $\X = \{\boldsymbol x_i\}_{i=1}^{\tilde m} \subset\Omega$ and boundary points $\Y = \{\boldsymbol z_j\}_{j=1}^{\bar m}\subset \Gamma$, we define the discrete loss
\begin{align}\label{6.4}
	\mathcal{L}^*_{\PR}(\boldsymbol{v})
	= \sum_{l=1}^{d}\left[\frac{1}{\tilde{m}}\sum_{i=1}^{\tilde{m}}\left|\nabla\times \left(-\nu\Delta v(\boldsymbol{x}_i) + (\beta\cdot\nabla) v(\boldsymbol{x}_i) +\sigma v(\boldsymbol{x}_i)- f(\boldsymbol{x}_i)\right) \right|^\gamma \right]^{2/\gamma}\notag\\ 
	+ \frac{1}{\bar{m}}\sum_{l=1}^{d}\sum_{j=1}^{\bar{m}}|g^l(\boldsymbol{z}_j) - v^l(\boldsymbol{z}_j)|^2+
	\frac{1}{\bar m^2}\sum_{l=1}^{d}
	\sum_{\substack{i,j=1 \\ i\neq j}}^{\bar m}
	\frac{\lvert [g - v](\boldsymbol{z}_i) - [g -v](\boldsymbol{z}_j) \rvert^2}{\lvert \boldsymbol{z}_i - \boldsymbol{z}_j\rvert^d},
\end{align}
where $\gamma \in [1,2]$. This loss provides a norm-equivalent discretization of the continuous functional~\eqref{6.3}, based on the Sobolev embeddings into $\boldsymbol{H}^{-2}(\Omega)$. In particular, for spatial dimension $d=2$, we use the embedding $\boldsymbol{L}^\gamma(\Omega)\hookrightarrow \boldsymbol{H}^{-2}(\Omega)$. Consequently, it yields CPINNs approximations of the pressure-robust formulation.

\subsection{Velocity recovery from the pressure-robust loss}\label{Sec_6.4}
Let $\tilde{\boldsymbol{u}}$ denote the minimizer of the proposed loss functional $\L_{\PR}$. Then $\tilde{\boldsymbol{u}}$ satisfies the modified momentum equation given by
\begin{align}\label{6.5}
	- \nu \Delta \tilde {\boldsymbol{u}}+ (\boldsymbol{\beta} \cdot \nabla) \tilde{\boldsymbol{u}} + \sigma \tilde{\boldsymbol{u}}=\tilde{\boldsymbol{f}},\quad \curl(\tilde{\boldsymbol{f}}-\boldsymbol{f})=0,
\end{align}
with condition $\tilde {\boldsymbol{u}} =g$ over the boundary $\Gamma$. The later condition implies that $(\tilde{\boldsymbol{f}}-\boldsymbol{f})$ is a gradient field, and hence by uniqueness of the Helmholtz–Hodge decomposition
\begin{align*}
	\tilde{\boldsymbol{f}}= \mathbb{P}(\boldsymbol{f}),
\end{align*}
where $\mathbb{P}$ denotes the Helmholtz–Hodge projector onto the divergence-free subspace. As a result, $\tilde{\boldsymbol{u}}$ coincides with the exact velocity solution $\boldsymbol{u}$ of the Oseen problem. Thus, the problem~\eqref{6.5} can be reformulated as
\begin{align}\label{6.6}
	\begin{cases}
		- \nu \Delta \boldsymbol{u} + (\boldsymbol{\beta} \cdot \nabla) \boldsymbol{u} + \sigma \boldsymbol{u}\hspace{-0.4cm}&=\mathbb{P}(\boldsymbol{f}) \quad \text{in } \Omega, \\
		\hspace{3.2cm}\boldsymbol{u} &= \boldsymbol{g} \quad \text{on } \Gamma.
	\end{cases}
\end{align}

From the standard well posedness theory for the Oseen problem~\cite[Chapter~VIII]{galdi2011introduction}  there exist a unique solution $\boldsymbol{u}\in \boldsymbol{V}_0(\Omega)$ of~\eqref{6.6} which satisfies the estmate
\begin{align}\label{6.7}
	\|\boldsymbol{u}\|_{\boldsymbol{H}^1(\Omega)} \lesssim \|\mathbb{P}(\boldsymbol{f}) \|_{\boldsymbol{H}^{-1}(\Omega)} + \|\boldsymbol{g}\|_{\boldsymbol{H}^{1/2}(\Gamma)}.
\end{align}
For converse of this estimate, after taking the $\boldsymbol{H}^{-1}(\Omega)$ norm of $\mathbb{P}(\boldsymbol{f})$, and applying the definition of dual norm and Poincar$\acute{e}$ inequality, we get
\begin{align}\label{6.8}
	\|\mathbb{P}(\boldsymbol{f}) \|_{\boldsymbol{H}^{-1}(\Omega)} = \|	- \nu \Delta \boldsymbol{u} + (\boldsymbol{\beta} \cdot \nabla) \boldsymbol{u} + \sigma \boldsymbol{u}\|_{\boldsymbol{H}^{-1}(\Omega)}
	\lesssim \|\boldsymbol{u}\|_{\boldsymbol{H}^1(\Omega)}.
\end{align}
Similarly, by continuity of the trace operator we get
\begin{align}\label{6.9}
	\|\boldsymbol{g}\|_{\boldsymbol{H}^{1/2}(\Gamma)} = \|\operatorname{Tr}(\boldsymbol{u})\|_{\boldsymbol{H}^{1/2}(\Gamma)}\lesssim \|\boldsymbol{u}\|_{\boldsymbol{H}^1(\Omega)}.
\end{align}
Combining the estimate~\eqref{6.7} with~\eqref{6.8} and~\eqref{6.8} for the upper bound and using the definition of dual norm and trace norm for the lower bound, we obtain the norm equivalence
\begin{align}\label{6.10}
	\|\boldsymbol{u}\|_{\boldsymbol{H}^1(\Omega)} \asymp \|\mathbb{P}(\boldsymbol{f}) \|_{\boldsymbol{H}^{-1}(\Omega)} + \|\boldsymbol{g}\|_{\boldsymbol{H}^{1/2}(\Gamma)}.
\end{align}

\subsection{Pressure-robust optimal recovery of the velocity \texorpdfstring{$\boldsymbol{u}$}{u}}\label{Sec_7}
We now establish the optimal recovery properties of the velocity field obtained from the pressure-robust formulation. Recall that the velocity $\boldsymbol{u}$ is recovered as the unique minimizer of the consistent loss functional $\L^*_{\PR}$ defined in~\eqref{6.4}. By construction, this formulation eliminates the pressure variable and enforces consistency only with the divergence-free component of the forcing term.

Let the model classes $\tilde\F$ and $\G$ be defined as
\[\tilde \F = U(\boldsymbol B^s_{pq}(\Omega)),\quad \G = \operatorname{Tr}(U(\boldsymbol B^{\bar{s}}_{\bar{p}\bar{q}}(\Omega))),\]
with the smoothness assumptions
\begin{align*}
	s>\frac{d}{p},\ \bar{s}>\frac{d}{\bar{p}},\ 0<p,q\leq\infty, \ 0<\bar{p}\leq 2, \ 0<\bar{q}\leq \infty.
\end{align*}
These conditions ensure that Helmholtz–Hodge projection $\mathbb{P}(\f)$ of the forcing term and boundary data $\boldsymbol{g}$ are continuous and that the Oseen problem admits a unique velocity solution $u\in \boldsymbol{H}^1(\Omega)$. For the given functional settings, define the velocity solution class
\[\U_{\PR} =\{\boldsymbol{u}\in [C(\Omega)]^d \mid \boldsymbol{u} \text{ solves } \eqref{6.6} \text{ for some } \mathbb{P}(\f)\in \tilde\F, \g \in \G \}.\]
From the pressure-robust reformulation, the recovered velocity depends only on the Helmholtz–Hodge projection $\mathbb{P}(\f)$ of the forcing term. In particular, from the equivalence~\eqref{6.10} for any two admissible data pairs $(\mathbb{P}(\f_1),\g_1)$ and $(\mathbb{P}(\f_2),\g_2)$, the corresponding velocities $ \boldsymbol{u}_1, \boldsymbol{u}_2 \in \U_{\PR}$ satisfy the stability estimate
\begin{align}\label{5.9}
	\|\boldsymbol{u}_1-\boldsymbol{u}_2\|_{\boldsymbol{H}^1(\Omega)} \asymp \|\mathbb{P}(\boldsymbol{f}_1)-\mathbb{P}(\boldsymbol{f}_2)\|_{\boldsymbol{H}^{-1}(\Omega)} + \|\boldsymbol{g}_1-\boldsymbol{g}_2\|_{\boldsymbol{H}^{1/2}(\Gamma)}.
\end{align}

\subsubsection{Optimal recovery rate for \texorpdfstring{$\boldsymbol{u}$}{u}}
Let the interior sampling points $\X\subset\Omega$ and the boundary sampling points $\Y\subset\Gamma$ are given, with total sampling budget $m = |\X| + |\Y| = \tilde{m}+\bar{m}$. Using the optimal recovery rates for $\mathbb{P}(\boldsymbol{f})$ in $\boldsymbol{H}^{-1}(\Omega)$ (Theorem~\ref{thm4.1}) and for $\boldsymbol{g}$ in $\boldsymbol{H}^{1/2}(\Gamma)$ (Theorem~\ref{thm4.2}), we obtain 
\begin{align}\label{6.12}
	R^*(\U_{\PR},\X,\Y)_{ \boldsymbol{H}^1(\Omega)}
	\asymp
	R^*(\tilde\F,\X)_{ \boldsymbol{H}^{-1}(\Omega)} + R^*(\G,\Y)_{ \boldsymbol{H}^{1/2}(\Gamma)} .
\end{align}
Optimizing it over all admissible sampling densities yields the following result.
\begin{theorem}\label{thm7.1}
	Let $\Omega=(0,1)^d$ and assume the model classes $\tilde\F, \G$ and $\U_{\PR}$ as defined above. Then the optimal recovery rate of the velocity field in the $\boldsymbol{H}^1(\Omega)$ norm satisfies
	\[	R^*_m(\U_{\PR},\X,\Y)_{ \boldsymbol{H}^1(\Omega)} \asymp m^{-\min\{\alpha,\beta\}}, \quad m\geq 2,   \]
	for $d=2$ with $p>1$. For the case of $d=2$ and $0<p\leq 1$, the rate becomes
	\[ m^{-\min\{\alpha,\beta\}}	\lesssim R^*_m(\U_{\PR},\X,\Y)_{ \boldsymbol{H}^1(\Omega)} \lesssim \log(m)m^{-\alpha} + m^{-\beta}, \]
	where $\alpha$ and $\beta$ are as defined in~\eqref{4.9} and~\eqref{4.24}, corresponding to the divergence-free component $\mathbb{P}(\f)$ of the forcing term and boundary data $\g$, respectively, and all constants appearing in the equivalence are independent of $m$.
\end{theorem}
\begin{proof}
	For the balanced allocation  of collocation points we take $\tilde m = \bar m = m/2$. From Theorem~\ref{thm4.1}, the optimal recovery rate for the divergence-free forcing
	satisfies
	\[
	R^*_{\tilde m}(\tilde \F,\boldsymbol{H}^{-1}(\Omega))
	\;\lesssim\;
	\begin{cases}
		\tilde m^{-\alpha}, & d=2,\ p>1,\\[2mm]
		\log(\tilde m)\,\tilde m^{-\alpha}, & d=2,\ 0<p\le 1 .
	\end{cases}
	\]
	From Theorem~\ref{thm4.2}, the boundary recovery satisfies
	\[
	R^*_{\bar m}(\G,\boldsymbol{H}^{1/2}(\Gamma)) \;\lesssim\; \bar m^{-\beta}.
	\]
	Substituting these bounds into \eqref{6.12} yields the upper bound
	\[
	R^*_m(\U_{\PR})_{\boldsymbol{H}^1(\Omega)}
	\;\lesssim\;
	\tilde m^{-\alpha} + \bar m^{-\beta}
	\;\lesssim\;
	m^{-\min\{\alpha,\beta\}},
	\]
	with the logarithmic modification in the case of $0<p\le1$. Let $\X$ and $\Y$ be arbitrary sampling sets with $|\X|+|\Y|=m$.
	Using the equal sampling $\tilde m = \bar m = m/2$ together with the lower bounds from
	Theorem~\ref{thm4.1} and Theorems~\ref{thm4.2} , we obtain
	\[
	R^*_m(\U_{\PR})_{\boldsymbol{H}^1(\Omega)}
	\;\gtrsim\;
	\inf_{\tilde m+\bar m=m}
	\left(
	\tilde m^{-\alpha} + \bar m^{-\beta}
	\right)\;\gtrsim\;
	m^{-\min\{\alpha,\beta\}} .
	\]
	Combining the upper and lower bounds proves the optimal recovery rate of the pressure-robust velocity in the $\boldsymbol{H}^1(\Omega)$ norm.
\end{proof}

\subsection{Pressure recovery}\label{presssure_recovery}
While the proposed formulation eliminates the pressure from the velocity learning process, the pressure field can be recovered once the velocity is known. Thus, for the given recovered velocity $\boldsymbol{u}$, and the modified forcing term $\bar{\boldsymbol{f}} :=  \boldsymbol{f} - \left(	- \nu \Delta {\boldsymbol{u}}+ (\boldsymbol{\beta} \cdot \nabla) {\boldsymbol{u}} + \sigma {\boldsymbol{u}}\right)$, the pressure gradient is obtained from the momentum equation with the normalization condition
\begin{align*}
	\nabla p &= \bar{\boldsymbol{f}} \ \text{in }\, \Omega, \\
	\int_{\Omega}^{}p\,dx&=0.
\end{align*}
This defines a well-posed first order problem for the pressure, ensuring consistent recovery without affecting the velocity approximation. Thus, from~\cite[Chapter~VIII]{galdi2011introduction} there exist a unique solution $p\in L^2_0(\Omega)$ of this problem satisfying the stability estimate
\begin{align}\label{6.13}
	\|p\|_{L^2(\Omega)} \lesssim \|\boldsymbol{f}\|_{\boldsymbol{H}^{-1}(\Omega)} + \|\boldsymbol{u}\|_{\boldsymbol{H}^1(\Omega)}.
\end{align}

To compute the pressure numerically, we introduce a neural network approximation $p_{\theta}\in \N_n^{p}$ to define the pressure by using CPINNs formulation. Thus the resulting optimization problem given by
\begin{align}
	p_{\theta} = \mathop{\arg\min}_{q \in L_0^2(\Omega)} \mathcal{L}_p(q),
\end{align} 
where the pressure recovery loss functional defined as
\begin{align}\label{5.13}
	\L_p(q) = \|\nabla q - \bar{\boldsymbol{f}}\|^2_{\boldsymbol{H}^{-1}(\Omega)}.
\end{align}

\subsubsection{Optimal recovery of the pressure $p$}
We now analyze the optimal recovery of the pressure field obtained a posteriori from the recovered velocity. As discussed in Section~\ref{Sec_7}, once the velocity $\boldsymbol u$ is known, the pressure gradient is uniquely determined from the momentum equation together with the normalization condition
\begin{align}\label{8.2}
	\nabla p = \bar{\boldsymbol{f}}  \ \text{in }\, \Omega,\quad  \int_{\Omega}^{}p\,dx=0.
\end{align}
Define the associated pressure solution class
\begin{align*}
	\mathscr{P}:= \{p\in L^2_0(\Omega) \mid p \text{ solves } \eqref{8.2} \text{ for some } \boldsymbol f\in\F, \boldsymbol u\in \U_{\PR}\}.
\end{align*}

Let $\boldsymbol u_1, \boldsymbol u_2\in \boldsymbol{H}^1(\Omega)$ be two velocity fields and let $p_1,p_2\in L_0^2(\Omega)$ denote the corresponding pressures recovered as
\begin{align}
	\nabla p_i = \boldsymbol{f}_i - \left(- \nu \Delta \boldsymbol{u}_{i}+ (\boldsymbol{\beta} \cdot \nabla) \boldsymbol{u}_{i} + \sigma \boldsymbol{u}_{i}\right),\quad i = 1,2.
\end{align}
From the stability estimate~\eqref{6.13}, we get
\begin{align}\label{8.7}
	\|p_1-p_2\|_{L^2(\Omega)} \lesssim \|\boldsymbol{f}_1- \boldsymbol{f}_2\|_{\boldsymbol{H}^{-1}(\Omega)} + \|\boldsymbol{u}_1-\boldsymbol{u}_2\|_{\boldsymbol{H}^1(\Omega)}.
\end{align}
Thus, from the optimal recovery estimates for $\boldsymbol{f}$ and $\boldsymbol{u}$, it follows that the pressure recovery error is dominated by the slower of the two rates. In particular, since the velocity $\boldsymbol{u}$ is recovered optimally in $\boldsymbol{H}^1(\Omega)$, the pressure recovery express the same asymptotic behavior.
\begin{theorem}
	Let $p\in \PP$ be recovered from the pressure-robust approximation as described in Section~\ref{presssure_recovery}. Under the assumptions of Theorem~\ref{thm7.1}, the optimal recovery rate of the pressure satisfies
	\begin{align}
		R_m^*(\PP)_{L^2(\Omega)} \lesssim m^{-\min\{\alpha,\beta\}}, \quad  m\geq 2,
	\end{align}
	for $d=2$ with $p>1$. When $0<p\leq 1$ the rate becomes
	\begin{align}
		R^*_m(\PP)_{L^2(\Omega)} \lesssim \log(m)m^{-\alpha} + m^{-\beta}.
	\end{align}
\end{theorem}
\begin{proof}
	Let $\U_{\PR}$ denote the pressure-robust velocity solution class and $\PP$ be the corresponding pressure class. For fixed sampling sets $\X\subset\Omega$ and $\Y\subset \Gamma$, inequality~\eqref{8.7} implies
	\begin{align}
		R^*(\PP,\X,\Y)_{L^2(\Omega)} \lesssim R^*(\F,\X)_{H^{-1}(\Omega)} + R^*(\U_{\PR},\X,\Y)_{H^1(\Omega)}.
	\end{align}
	Using the optimal recovery rate for the pressure-robust velocity class $\U_{\PR}$ from Theorem~\ref{thm7.1} and the optimal recovery rate for the forcing class $\tilde{\F}$ from Theorem~\ref{thm4.1}. Then, for the balanced allocation $\tilde m=\bar m = m/2$, we obtain
	\begin{align}
		R^*_m(\PP)_{L^2(\Omega)}\lesssim \tilde{m}^{-\alpha} + m^{-\min\{\alpha,\beta\}} \lesssim  m^{-\min\{\alpha,\beta\}},
	\end{align}
	with the usual logarithmic modification in the case $0<p\leq 1$.
\end{proof}

\begin{remark}
	A pressure-robust formulation can also be considered without enforcing the incompressibility constraint directly in the neural network architecture. In this case, the loss functional is defined as
	\begin{align}
		\L_{\PR}(\boldsymbol{v}) = \|\nabla \times \left(-\nu\Delta\boldsymbol{v} +  (\boldsymbol{\beta}\cdot\nabla) \boldsymbol{v} + \sigma \boldsymbol{v}-\boldsymbol{f} \right) \|^2_{\boldsymbol{H}^{-2}(\Omega)} + \|\nabla\cdot\boldsymbol{v}  \|^2_{L^2(\Omega)} \\\notag+ \|\boldsymbol{g}-\boldsymbol{v}\|^2_{\boldsymbol{H}^{1/2}(\Gamma)} .
	\end{align}
	The curl of the momentum residual eliminates the effect of gradient forces and therefore preserves pressure robustness. The velocity approximation $\boldsymbol{u} _{NN}$ is obtained directly from the neural network architecture and is not required to satisfy the div-free constraint. Once the velocity is recovered, the pressure field can be computed separately using the loss $\mathcal{L}_p$ defined in \eqref{5.13}. This results in a pressure-robust PINNs formulation in which the velocity learning stage does not impose the divergence-free condition explicitly.
\end{remark}
\subsection{Discrete residual control for pressure-robust CPINNs}
The proposed loss functionals $\L_{\PR}, \L_p$, formulated in terms of continuous norms, are not directly suitable for numerical implementation. To overcome this, we introduce corresponding computable discrete surrogates $\L^*_{\PR}$ and $\L^*_{p}$, constructed using only finitely many sampled evaluations of the problem data, such as the forcing term and the boundary conditions. We consider a set of interior collocation points $\X:= G_{k,r}\subset \Omega$ and a set of boundary points $\Y:= \overline{G}_{k,r}\subset\Gamma$ with corresponding cardinalities $\tilde{m}=|\X|$ and $\bar{m} = |\Y|$. Under suitable assumptions on $\F$ and $\G$, the exact solution $(\boldsymbol u,p)$ of the Oseen system belongs to a prescribed approximation class $\U\times\mathscr{P}$. Thus, the best possible recovery rate to recover solution is given by $\max\{\tilde{m}^{-s/d}, \bar{m}^{-(s-1)/(d-1)}\}$. In order to quantify the residual of an approximation, for any pair $(\boldsymbol v,q)$ we define the norm $\|\boldsymbol v\|_{\U}:= \max\{\|\nabla \times (- \nu \Delta \boldsymbol{v} + (\boldsymbol{\beta} \cdot \nabla) \boldsymbol{v} + \sigma \boldsymbol{v})\|_{\B}, \|\operatorname{Tr}(\boldsymbol v)\|_{\operatorname{Tr}(\overline{\B})}\}$ and $\|(\boldsymbol v,p)\|_{\mathscr{P}} = \max\{\|\boldsymbol v\|_{\U}, \|\nabla q\|_{\B}\}$. Using this construction, we will show that the approximation error in the velocity, measured in the $\boldsymbol H^1(\Omega)$-norm, is bounded in terms of the discrete loss functional $\L^*_{\PR}$, and the approximation error in the pressure in $L^2(\Omega)$-norm, is bounded by $\L^*_{\PR}(\boldsymbol v)$ and $\L^*_{p}(q)$. The result will hold under the regularity assumptions
that $(- \nu \Delta \boldsymbol{v} + (\boldsymbol{\beta} \cdot \nabla) \boldsymbol{v} + \sigma \boldsymbol{v})\in\B$ with $\boldsymbol v\in\overline{\B}$, and $q\in B$.
\begin{theorem}\label{thm_5.5}
	Let $\Omega$ be a simply-connected bounded domain, $(\boldsymbol u,p)$ be the exact solution of the pressure-robust formulation corresponding to data $\boldsymbol f\in \F:= U(\B)$, $\boldsymbol g\in \G$ as defined in~\eqref{eq:F-def}, and let $(\boldsymbol v,q)$ be an admissible approximation. Assume sampling set $\X\subset \Omega$, $\Y \subset \Gamma$ with sizes $\tilde{m}, \bar m$ respectively. Then, the discrete loss functional $L^*_{\PR}(\boldsymbol v)$ bounds the velocity error as
	\begin{align}
		\|\boldsymbol u- \boldsymbol v\|^2_{\boldsymbol{H}^1(\Omega)} \lesssim \L_{\PR}^*(\boldsymbol  v) + (\|\boldsymbol v\|_{\U}^2+1 )\mathfrak{R}_{\PR}(\tilde m, \bar{m}),
	\end{align}
	and the discrete pressure loss functional $\L_{p}^*(q)$ satisfies the estimate
	\begin{align}
		\|p-q\|^2_{L^2(\Omega)} \lesssim \L_{p}^*(q)+ \L^*_{\PR}(\boldsymbol  v) + (\|(\boldsymbol v,p)\|_{\mathscr{P}} ^2+1 )\mathfrak{R}_{\PR}(\tilde m, \bar{m}),
	\end{align}
	where the residual term is given by 
	\[\mathfrak{R}_{\PR}(\tilde m, \bar{m}) = \max\{\tilde{m}^{-2s/d}, \bar{m}^{-2(s-1)/(d-1)}\}.\]
	The constants that appear in equivalence will be independent of $\boldsymbol u,\boldsymbol v,\tilde m, \bar m$.
\end{theorem}
\begin{proof}We first prove the error bound for the velocity field, and subsequently derive the estimate for the pressure. Since $\boldsymbol u$ is divergence-free, it follows from~\cite{girault1986finite}, that for $d=2$ there exist a stream function $\phi$ such that $\boldsymbol u=\nabla\times \phi$. Similarly, for the admissible approximation $\boldsymbol v$ there exist $\tilde \phi$ such that $\boldsymbol v = \nabla\times \tilde\phi$. For this representation the problem~\eqref{1} will change as a biharmonic problem for $(\phi-\tilde{{\phi}})$. Thus, by the stability estimate for the stream-function formulation~\cite[Section~5.5, Chapetr~1]{girault1986finite}, it follows that
	\begin{align}\label{5.24}
		\|\phi-\tilde \phi\|_{\boldsymbol H^2(\Omega)} \leq \|\nabla\times (\boldsymbol f-\tilde {\boldsymbol f})\|_{H^{-2}(\Omega)} + \|\boldsymbol g-\tilde {\boldsymbol g}\|_{\boldsymbol H^{1/2}(\Gamma)}.
	\end{align}
	Using $\boldsymbol u-\boldsymbol v = \nabla\times( \phi-\tilde{{\phi}})$ in~\eqref{5.24}, we get
	\begin{align}\label{5.25}
		\|\boldsymbol u-\boldsymbol v\|_{\boldsymbol H^1(\Omega)}&= \|\nabla \times ( \phi-\tilde{\phi})\|_{\boldsymbol H^1(\Omega)}\lesssim \|\phi-\tilde{\phi}\|_{H^2(\Omega)}\notag\\
		&\lesssim  \|\nabla\times (\boldsymbol f-\tilde {\boldsymbol f})\|_{H^{-2}(\Omega)} + \|\boldsymbol g-\tilde {\boldsymbol g}\|_{\boldsymbol H^{1/2}(\Gamma)}.
	\end{align} 
	Finally, from the Sobolev embedding \(L^1(\Omega) \hookrightarrow H^{-2}(\Omega)\) for $d=2$, we obtain
	\begin{align}\label{5.26}
		\|\boldsymbol u-\boldsymbol v\|^2_{\boldsymbol H^1(\Omega)}&\lesssim \|\nabla \times (- \nu \Delta \boldsymbol{v} + (\boldsymbol{\beta} \cdot \nabla) \boldsymbol{v} + \sigma \boldsymbol{v}- \boldsymbol { f})\|^2_{{L}^{1}(\Omega)} + \|\boldsymbol v-\boldsymbol g\|^2_{\boldsymbol H^{1/2}(\Gamma)}\notag\\
		&\lesssim\|\nabla \times \R(\boldsymbol v)\|^2_{L^{1}(\Omega)}+ \|\boldsymbol v-\boldsymbol g\|^2_{\boldsymbol H^{1/2}(\Gamma)},
	\end{align}
	where $\R(\boldsymbol v) =  - \nu \Delta \boldsymbol{v} + (\boldsymbol{\beta} \cdot \nabla) \boldsymbol{v} + \sigma \boldsymbol{v}- \boldsymbol { f}$. From the norm equivalence of the continuous and discrete $L^{1}(\Omega)$-norm (see~\cite{BVPS}), we get the estimate
	\begin{align*}
		\|\nabla\times\R(\boldsymbol v) \|_{{L}^{1}(\Omega)} \leq  \|\nabla\times\R(\boldsymbol v)\|^*_{{L} ^{1}(\Omega)} +\|\nabla \times \R(\boldsymbol v) \|_{\B} \tilde{m}^{-\frac{s}{d}}. 
	\end{align*}
	Similarly, from~\cite[Theorem~6.4]{BVPS} the norm corresponding to the boundary term satisfies the estimate
	\[ \|\boldsymbol v-\boldsymbol g\|_{\boldsymbol{H}^{1/2}(\Gamma)} \leq \|\boldsymbol v-\boldsymbol g\|^*_{\boldsymbol{H}^{1/2}(\Gamma)} + \|\boldsymbol v-\boldsymbol g\|_{\operatorname{Tr}(\overline{\B})} \bar{m}^{-\frac{\bar{s}-1}{d-1}}. \]
	Collecting the above estimates and combining them with the stability bound~\eqref{5.26} yields
	\begin{align}
		\|\boldsymbol u- \boldsymbol v\|^2_{\boldsymbol{H}^1(\Omega)} &\leq  \|\nabla \times (- \nu \Delta \boldsymbol{v} + (\boldsymbol{\beta} \cdot \nabla) \boldsymbol{v} + \sigma \boldsymbol{v}- \boldsymbol { f})\|^{2,*}_{{L}^{1}(\Omega)} + \|\boldsymbol v-\boldsymbol g\|^{2,*}_{\boldsymbol H^{1/2}(\Gamma)}\notag\\
		&+  \|\nabla \times (- \nu \Delta \boldsymbol{v} + (\boldsymbol{\beta} \cdot \nabla) \boldsymbol{v} + \sigma \boldsymbol{v}- \boldsymbol { f})\|^{2}_{\B}\tilde{m}^{-\frac{2s}{d}} + \|\boldsymbol v-\boldsymbol g\|^{2}_{\operatorname{Tr}(\overline{\B})} \bar{m}^{-\frac{2(\bar{s}-1)}{d-1}}\notag\\
		&\lesssim \L^*_{\PR} + (\|\boldsymbol v\|_{\U}^2+1 )\mathfrak{R}_{\PR}(\tilde m, \bar{m}),
	\end{align}
	where the argument is completed by using the bound $\|\boldsymbol f\|_{\B}\leq1$ and $\|\boldsymbol g\|_{\operatorname{Tr}(\overline{\B})}\leq 1$, which concludes the proof for error bound of velocity approximation. Next we will prove, error bound for pressure approximation. From stability estimate~\eqref{8.7}, we get
	\begin{align*}
		\|p-q\|_{L^2(\Omega)} \lesssim \|\boldsymbol{f}- \hat {\boldsymbol{f}}\|_{\boldsymbol{H}^{-1}(\Omega)} + \|\boldsymbol{u}-\boldsymbol{v}\|_{\boldsymbol{H}^1(\Omega)}.
	\end{align*}
	Since $\hat{\boldsymbol{f}}:= \bar{\boldsymbol{f}} + \left(- \nu \Delta {\boldsymbol{u}}+ (\boldsymbol{\beta} \cdot \nabla){\boldsymbol{u}} + \sigma {\boldsymbol{u}}\right)$ as given in Section~\ref{presssure_recovery}. Then from the previous estimate this implies
	\begin{align}
		\|p-q\|^2_{L^2(\Omega)} &\lesssim \|\boldsymbol{f}- \bar{\boldsymbol{f}} - \left(	- \nu \Delta {\boldsymbol{u}}+ (\boldsymbol{\beta} \cdot \nabla){\boldsymbol{u}} + \sigma {\boldsymbol{u}}\right)\|^2_{\boldsymbol{H}^{-1}(\Omega)} + \|\boldsymbol{u}-\boldsymbol{v}\|^2_{\boldsymbol{H}^1(\Omega)}\notag\\
		&\lesssim \|\nabla q - \bar{\boldsymbol{f}}\|^2_{\boldsymbol{H}^{-1}(\Omega)} + \|\boldsymbol{u}-\boldsymbol{v}\|^2_{\boldsymbol{H}^1(\Omega)}.
	\end{align}
	Using the Sobolev embedding $L^\tau(\Omega)\hookrightarrow H^{-1}(\Omega)$ for $\tau >1$, and the norm equivalence of continuous and discrete $L^{\tau}$-norms, given in~\cite[Lemma~6.1]{BVPS}, it follows the estimate
	\begin{align}
		\|p-q\|^2_{L^2(\Omega)} &\lesssim  \|\nabla q - \bar{\boldsymbol{f}}\|^2_{\boldsymbol{L}^{\tau}(\Omega)} + \|\boldsymbol{u}-\boldsymbol{v}\|^2_{\boldsymbol{H}^1(\Omega)}\notag\\
		&\lesssim  \|\nabla q - \bar{\boldsymbol{f}}\|^{2,*}_{\boldsymbol{L}^{\tau}(\Omega)} + \|\nabla q - \bar{\boldsymbol{f}}\|^{2}_{\B}\tilde{m}^{-s/d} + \L^*_{\PR} + (\|\boldsymbol v\|_{\U}^2+1 )\mathfrak{R}_{\PR}(\tilde m, \bar{m}),\notag\\			
		&\lesssim \L^*_{p}(q)+ \L^*_{\PR}(\boldsymbol{ v}) + (\|(\boldsymbol v,p)\|_{\mathscr P}^2+1 )\mathfrak{R}_{\PR}(\tilde m, \bar{m}). 
	\end{align}
	This completes the proof for the pressure error estimate and hence concludes the proof.
\end{proof}

%

\section{Numerical experiments}\label{Sec_6}
In this section, we present numerical experiments to show the practical ability of the proposed CPINNs and pressure-robust CPINNs formulations for the stationary Oseen equations. The experiments are designed to validate the theoretical results and to demonstrate the accuracy and stability of the proposed methods. We consider benchmark examples with known analytical solutions and construct the corresponding forcing and boundary data accordingly.
\subsection{Example 1}
In this example we consider the stationary Oseen problem on the unit square $\Omega = (0,1)^2$ with constant coefficients $\nu = 1,\, \beta = (1,1)^T$, and $\sigma = 1$. We construct a smooth divergence-free velocity field by a smooth function
\begin{align}
	\psi(x,y) = \sin^2(x)\sin(y)\cos(y)\quad \text{in}\ \Omega,
\end{align}
and define the velocity field as
\begin{align*}
	u = (u_1,u_2) = \left(\frac{\partial \psi}{\partial y}, -\frac{\partial \psi}{\partial x}\right).
\end{align*}
We choose the pressure field $p(x,y) = \sin(\pi x)\cos(\pi y)$. Both the velocity and pressure are analytic in $\Omega$. The forcing term $f$ and boundary data $g$ are defined from the Oseen equation~\eqref{1}. We solve this example using three different neural-network formulations. 

In the standard PINNs formulation, both the velocity and pressure are approximated simultaneously by neural networks $(u_{\theta}, p_{\theta})$. The loss functional includes the momentum residual, the divergence constraint, and the boundary mismatch term. Since the pressure gradient appears explicitly in the momentum equation, any error in the pressure approximation may directly affect the velocity accuracy. Although this approach is consistent and convergent for smooth solutions, it is not pressure-robust, which means that strong or oscillatory pressure gradients can degrade the velocity approximation.
In the divergence-free formulation, the velocity approximation is restricted to the solenoidal space through a divergence-free network architecture. The incompressibility condition is thus satisfied exactly, and the loss contains only the momentum residual and boundary term. This improves structural consistency and typically reduces velocity error compared to the standard PINNs method. However, since the pressure still appears in the residual, the formulation remains sensitive to pressure gradients and is therefore not fully pressure-robust.

In case of pressure-robust formulation, the velocity is learned independently of the pressure variable by penalizing only the curl of the momentum residual together with the boundary mismatch. Consequently, the velocity depends solely on the divergence-free component of the forcing term. The pressure is then recovered in a post-processing step mentioned in Section~\ref{presssure_recovery}. For smooth test cases, all three approaches converge as the collocation density increases; however, the pressure-robust formulation yields the smallest velocity error in the $H^1(\Omega)$-norm, while the divergence-free method provides intermediate improvement over the standard PINNs approach. The corresponding errors over different numbers of collocation points are reported in Table~\ref{table_1}, Table~\ref{table_2} and Table~\ref{table_3} respectively.

\begin{table}[H]
	\centering
	\footnotesize
	\renewcommand{\arraystretch}{1.2}
	\begin{tabular}{|c|c c c c|}
		\hline
		
		& \multicolumn{4}{c|}{Primal formulation} \\
		\cline{2-5}
		Grid 
		& Velocity & Pressure &  $\|\nabla\cdot\boldsymbol{u}\|^*_{L^{\infty}}$  & Loss \\
		Points 
		& Error (\%) & Error (\%) & $(10^{-2})$ &  $(10^{-4})$\\
		
		$N$ & $\mathscr{P}$ / $\mathscr{C}$ & $\mathscr{P}$ / $\mathscr{C}$ & $\mathscr{P}$ / $\mathscr{C}$ & $\L_{}$ / $\L_{}^*$ \\
		\hline
		$5$  & 9.11 / 1.91 & 29.26 / 4.83 & 3.63 / 1.38 & 4.66 / 1.55 \\
		$10$ & 4.34 / 2.54 & 8.69 / 7.47 & 0.94 / 0.14 & 2.65 / 2.76 \\
		$15$ & 4.36 / 1.46& 9.55 / 4.55 & 0.94 / 0.14 & 3.26 / 2.57 \\
		$20$ & 3.34 / 1.50 & 8.64 / 3.49 & 0.24 / 0.10 & 3.69 / 2.78 \\
		$25$ & 3.12 / 1.41 & 8.37 / 3.21 & 0.27 / 0.12 & 2.67 / 2.12 \\
		$30$ & 3.28 / 1.22 & 8.82 / 3.12 & 0.16 / 0.16 & 3.14 / 2.98 \\
		\hline
	\end{tabular}
	\caption{Error comparison of solutions obtained using CPINNs ($\mathscr{C}$) and standard PINNs ($\mathscr{P}$) for primal formulation.}
	\label{table_1}
\end{table}

\begin{table}[H]
	\centering
	\footnotesize
	\renewcommand{\arraystretch}{1.2}
	\begin{tabular}{|c|c c c c|}
		\hline
		
		& \multicolumn{4}{c|}{Div-free formulation} \\
		\cline{2-5}
		Grid 
		& Velocity & Pressure &  $\|\nabla\cdot\boldsymbol{u}\|^*_{L^{\infty}}$  & Loss \\
		Points 
		& Error (\%) & Error (\%) & $(10^{-7})$ &  $(10^{-3})$\\
		
		$N$ & $\mathscr{P}$ / $\mathscr{C}$ & $\mathscr{P}$ / $\mathscr{C}$ & $\mathscr{P}$ / $\mathscr{C}$ & $\L_{}$ / $\L_{}^*$ \\
		\hline
		$5$  & 12.50 / 4.78 & 33.80 / 16.90 & 10.5 / 10.1 & 9.34 / 1.02 \\
		$10$ & 15.74 / 4.90 & 33.39 / 13.16 & 8.34 / 7.15 & 3.18 / 1.69 \\
		$15$ & 8.38 / 4.73 & 20.12 / 10.34 & 7.74 / 6.25 & 1.11 / 2.06 \\
		$20$ & 12.73 / 3.83 & 25.33 / 10.39 & 5.66 / 5.66 & 2.10 / 2.49 \\
		$25$ & 10.23 / 3.23 & 18.12 / 8.32 & 5.12 / 4.34 & 2.38 / 1.41 \\
		$30$ & 10.29 / 3.22 & 18.34 / 8.11 & 4.45 / 4.09 & 2.19 / 1.87 \\
		\hline
	\end{tabular}
	\caption{Error comparison of solutions obtained using CPINNs ($\mathscr{C}$) and standard PINNs ($\mathscr{P}$) for divergence free formulation.}
	\label{table_2}
\end{table}

\begin{table}[H]
	\centering
	\footnotesize
	\renewcommand{\arraystretch}{1.2}
	\begin{tabular}{|c|c c c c c|}
		\hline
		
		& \multicolumn{5}{c|}{Pressure-robust formulation} \\
		\cline{2-6}
		Grid 
		& Velocity &  $\|\nabla\cdot\boldsymbol{u}\|^*_{L^{\infty}}$  &Velocity  & Pressure & Pressure  \\
		Points 
		& Error (\%)  & $(10^{-7})$ & Loss $(10^{-4})$ & Error (\%) &  Loss $(10^{-3})$\\
		
		$N$ & $\mathscr{P}$ / $\mathscr{C}$ & $\mathscr{P}$ / $\mathscr{C}$ & $\L_{}$ / $\L_{}^*$ & $\mathscr{P}$ / $\mathscr{C}$ & $\L_{}$ / $\L_{}^*$\\
		\hline
		$5$  & 11.37 / 3.59  & 7.74 / 7.74 & 4.66 / 1.55 & 21.35 / 7.80 & 2.36 / 1.09 \\
		$10$ & 5.51 / 2.66 & 9.53 / 8.34 & 2.65 / 2.76 & 6.94 / 6.35 & 2.95 / 2.24 \\
		$15$ & 4.14 / 1.77  & 8.34 / 6.55 & 3.26 / 2.57 & 6.53 / 5.62 & 2.33 / 1.72 \\
		$20$ & 4.67 / 1.57  & 8.29 / 6.78 & 3.69 / 2.78 & 4.97 / 2.23 & 1.70 / 1.22 \\
		$25$ & 3.58 / 1.34  & 7.22 / 6.43 & 2.67 / 2.12 & 4.66 / 2.18 & 1.89 / 1.53 \\
		$30$ & 3.92 / 1.16  & 6.75 / 1.03 & 3.14 / 2.98 & 3.82 / 1.90 & 1.11 / 1.12 \\
		\hline
	\end{tabular}
	\caption{Error comparison of solutions obtained using CPINNs ($\mathscr{C}$) and standard PINNs ($\mathscr{P}$) for pressure-robust formulation.}
	\label{table_3}
\end{table}

%
%
%
%
%
%
%

\subsection{Example 2} 
We consider a no-flow benchmark problem adapted from the classical example of~\cite{john2017divergence}
in order to examine the proposed pressure-robust CPINNs formulation for the stationary Oseen system. Let $\Omega = (0,1)^2$ and consider the Oseen equation~\eqref{1}. We choose constant coefficients $\nu = 1$, $\sigma = 0$ and $\beta = (1,1)^T$. The forcing term is defined by 
\begin{align*}
f(x,y) = (0, \operatorname{Ra}(1-y+3y^2)), \quad \operatorname{Ra}>0.
\end{align*}
Define the pressure field $$p=  \operatorname{Ra}\left(y^3-\frac{y^2}{2}+y-\frac{7}{12}\right).$$
A direct computation shows that $\nabla p=f$, and therefore the exact velocity satisfies $\boldsymbol u=0$. Consequently, for all values of the Rayleigh parameter $\operatorname{Ra}$, the velocity remains identically zero, while the pressure scales linearly with $\operatorname{Ra}$. This example highlights the fundamental invariance property of incompressible flows, which modifies the forcing term by a gradient field that affects only the pressure, and the velocity remains unchanged. 

To assess the numerical behavior, we compute solutions for $\operatorname{Ra} = 1, 10^2, 10^4, 10^6$, and measure the quantity $\|\nabla \boldsymbol u_{NN}\|_{\boldsymbol L^2(\Omega)}$. Since the exact velocity vanishes, this norm directly represents the velocity error. For non-pressure robust formulations, gradient forcing typically contaminates the discrete velocity, leading to errors proportional to the magnitude of the pressure, i.e., $\|\nabla \boldsymbol u_{NN}\|_{\boldsymbol L^2(\Omega)}\sim \operatorname{Ra}$, as observed in~\cite{john2017divergence}. In contrast, the proposed pressure-robust CPINNs formulation is derived from a consistent residual norm that mimics the continuous stability estimate of the Oseen system (derived in Section~\ref{Sec_5}), and therefore preserves the separation between divergence-free and irrotational components of the forcing. This example thus provides a strong test case of the proposed pressure-robust CPINNs formulation and demonstrates the accuracy over the standard formulation.

\begin{figure}[H]
\centering
\caption{Convergence plots for the solution using PINNs (left) and CPINNs (right) for general formulation.}
\label{fig1}
\subfloat{\includegraphics[width=0.48\textwidth]{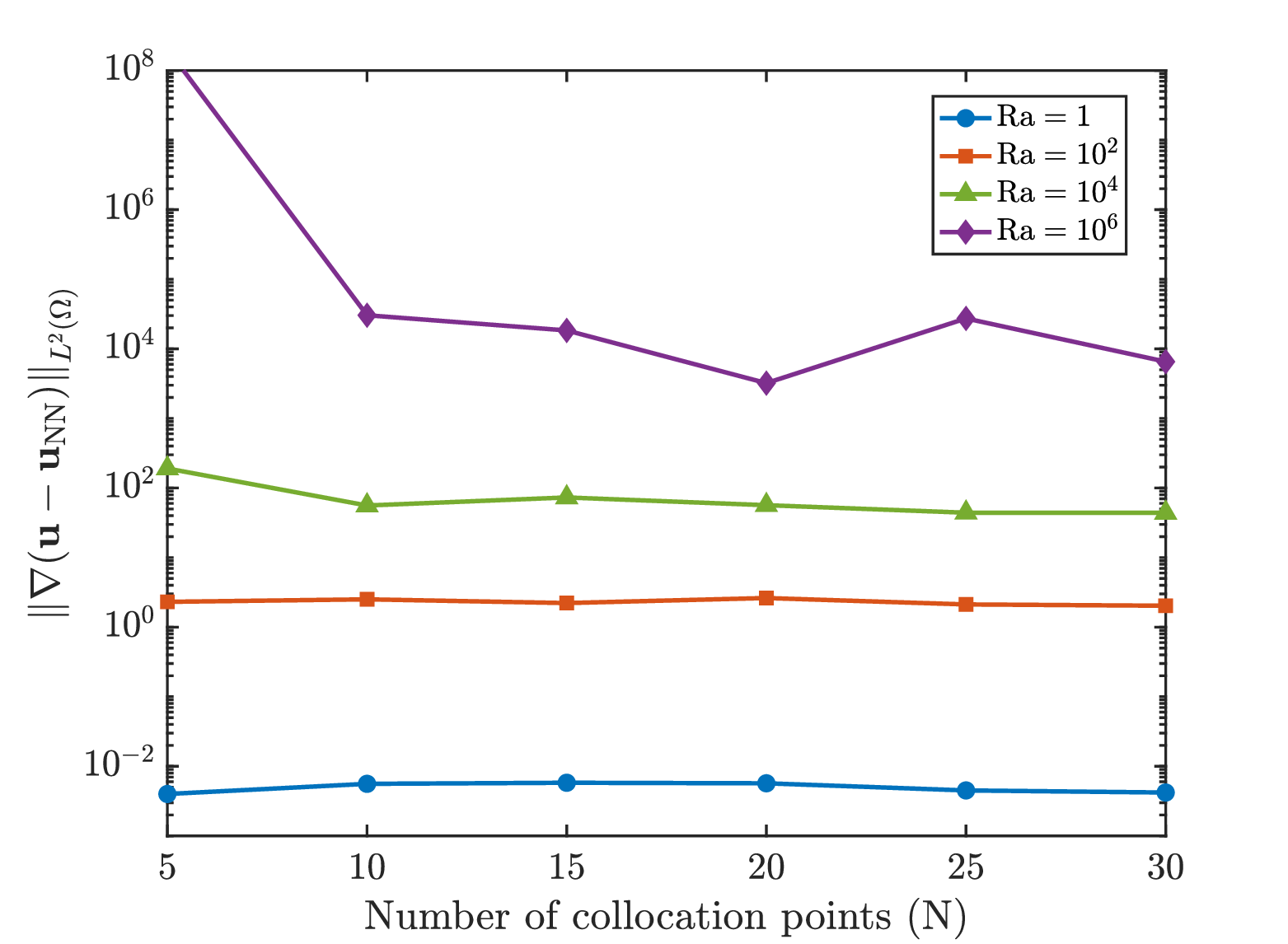}}
\hspace{0.3cm}
\subfloat{\includegraphics[width=0.48\textwidth]{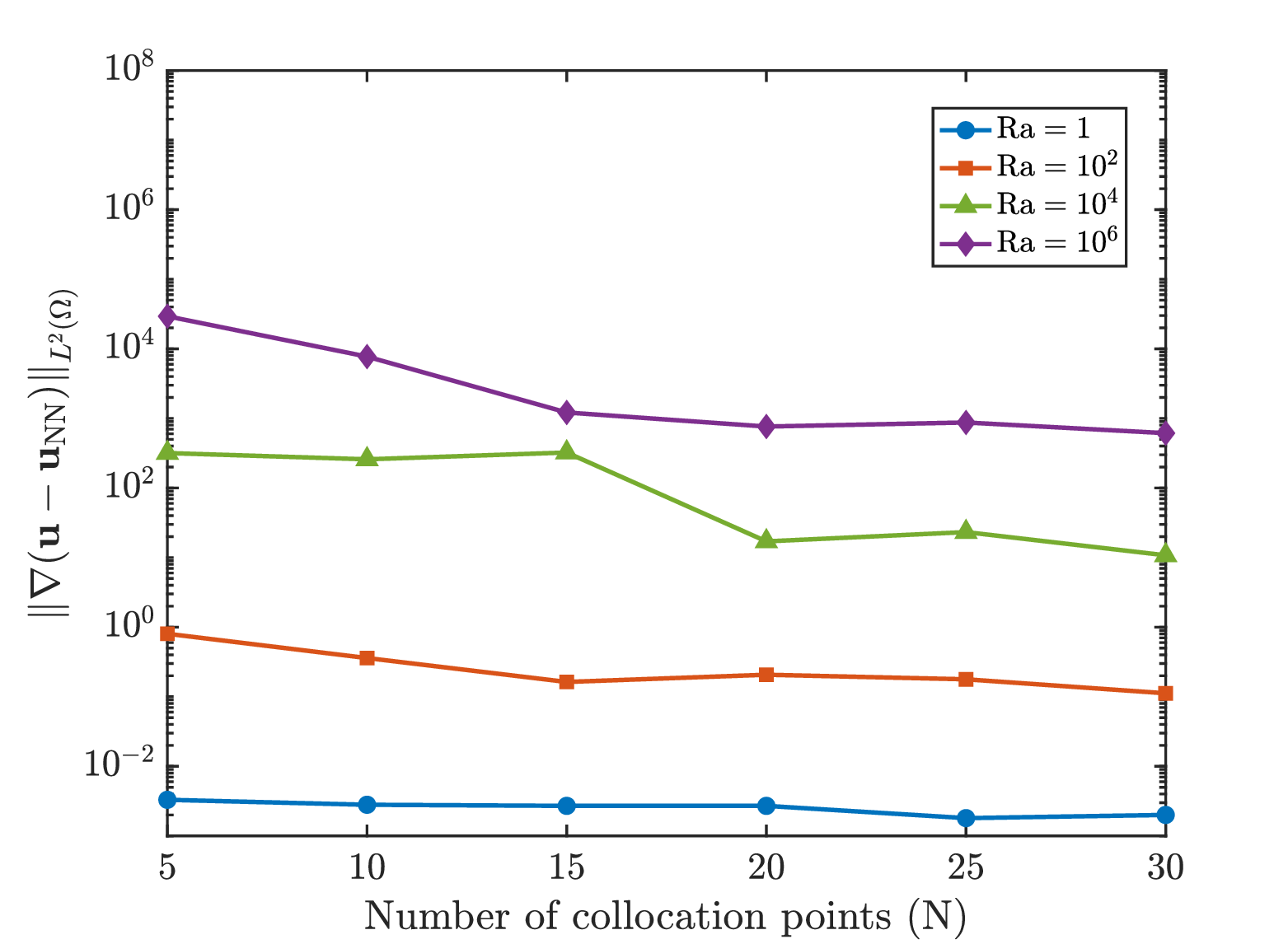}}
\end{figure}

\begin{figure}[H]
\centering
\caption{Convergence plots for the solution using PINNs (left) and CPINNs (right) for div-free formulation.}
\label{fig2}
\subfloat{\includegraphics[width=0.48\textwidth]{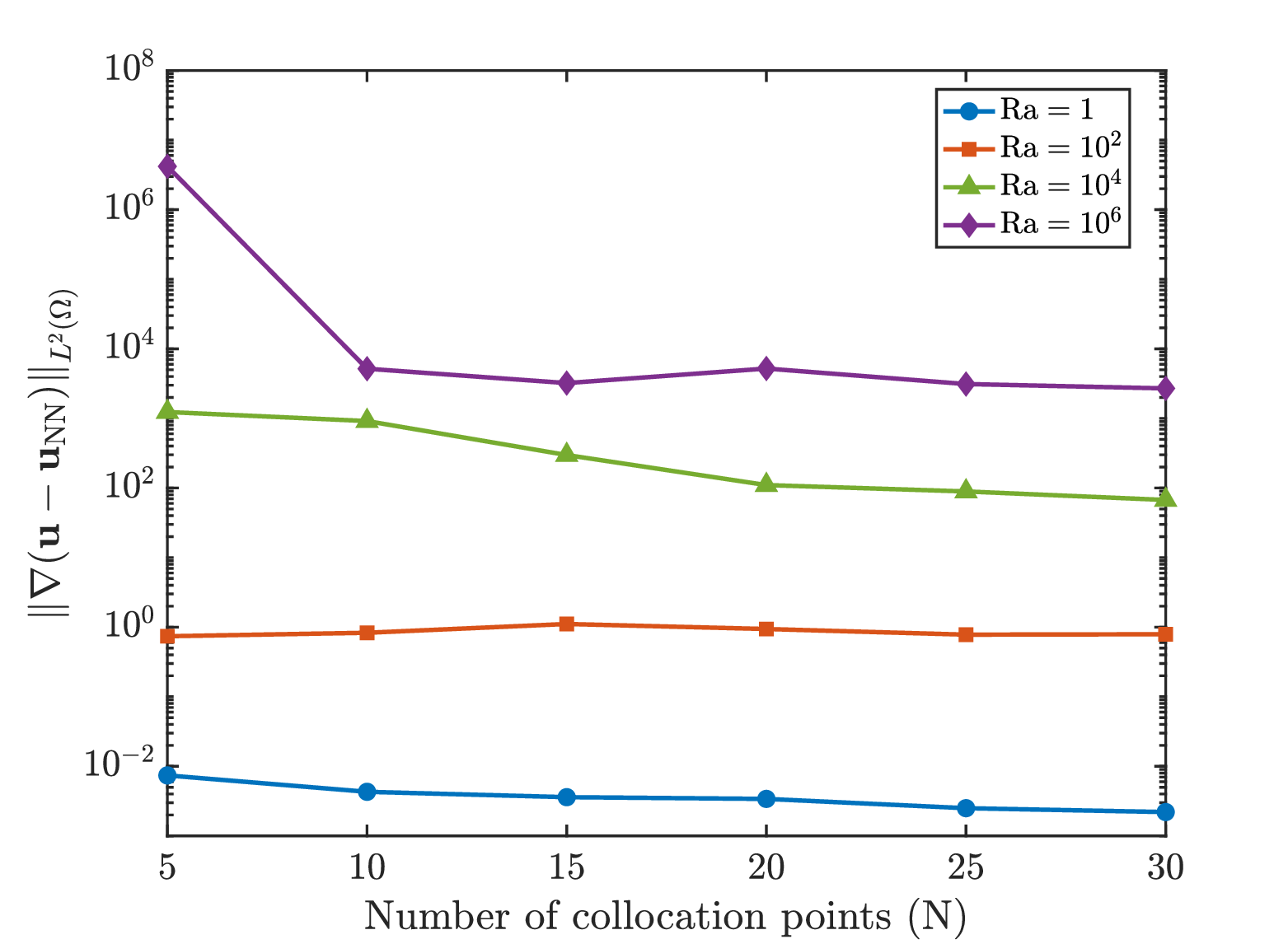}}
\hspace{0.3cm}
\subfloat{\includegraphics[width=0.48\textwidth]{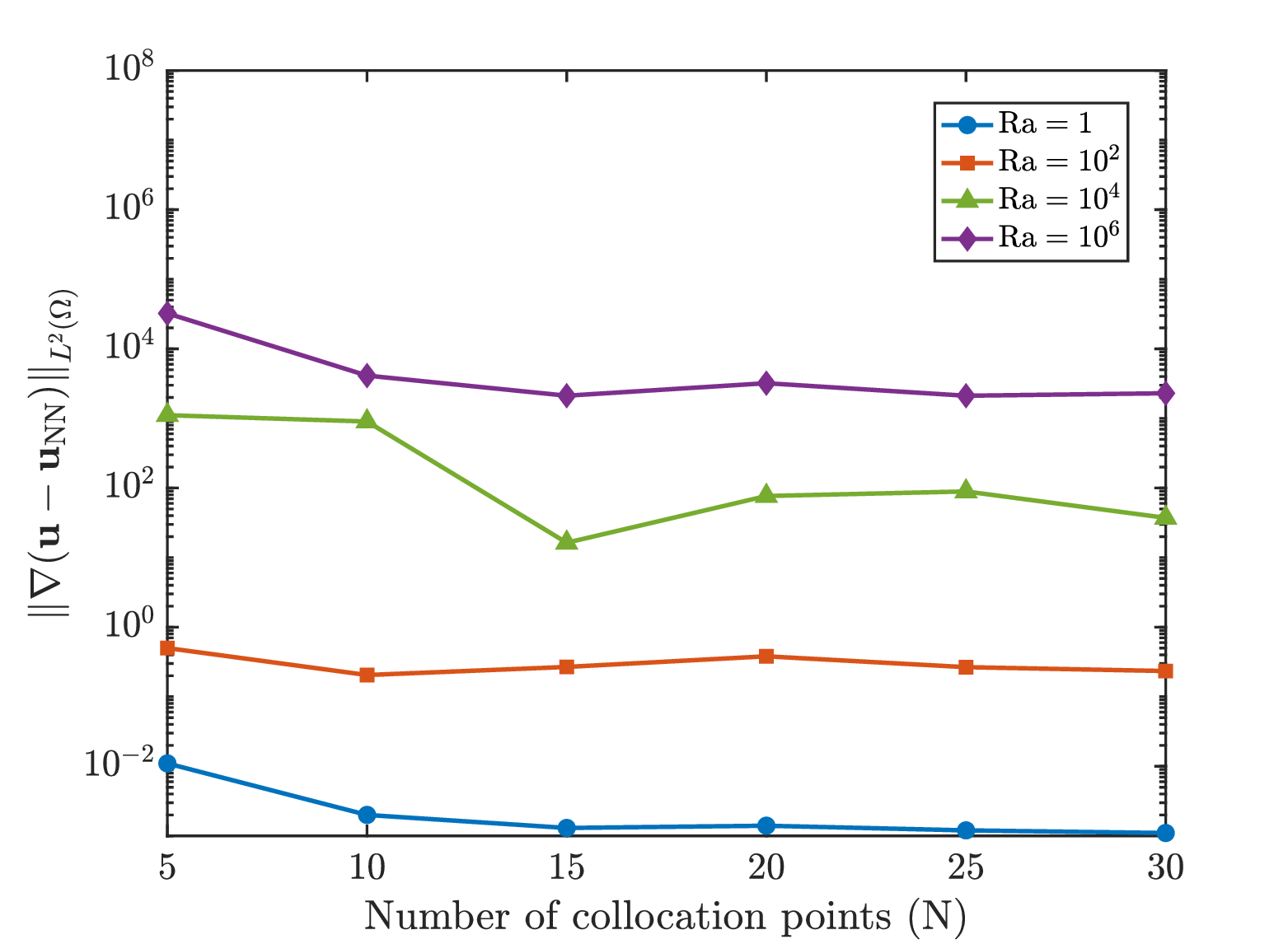}}
\end{figure}

\begin{figure}[H]
\centering
\caption{Convergence plots for the solution using PINNs (left) and CPINNs (right) for div-free pressure-robust formulation.}
\label{fig3}
\subfloat{\includegraphics[width=0.48\textwidth]{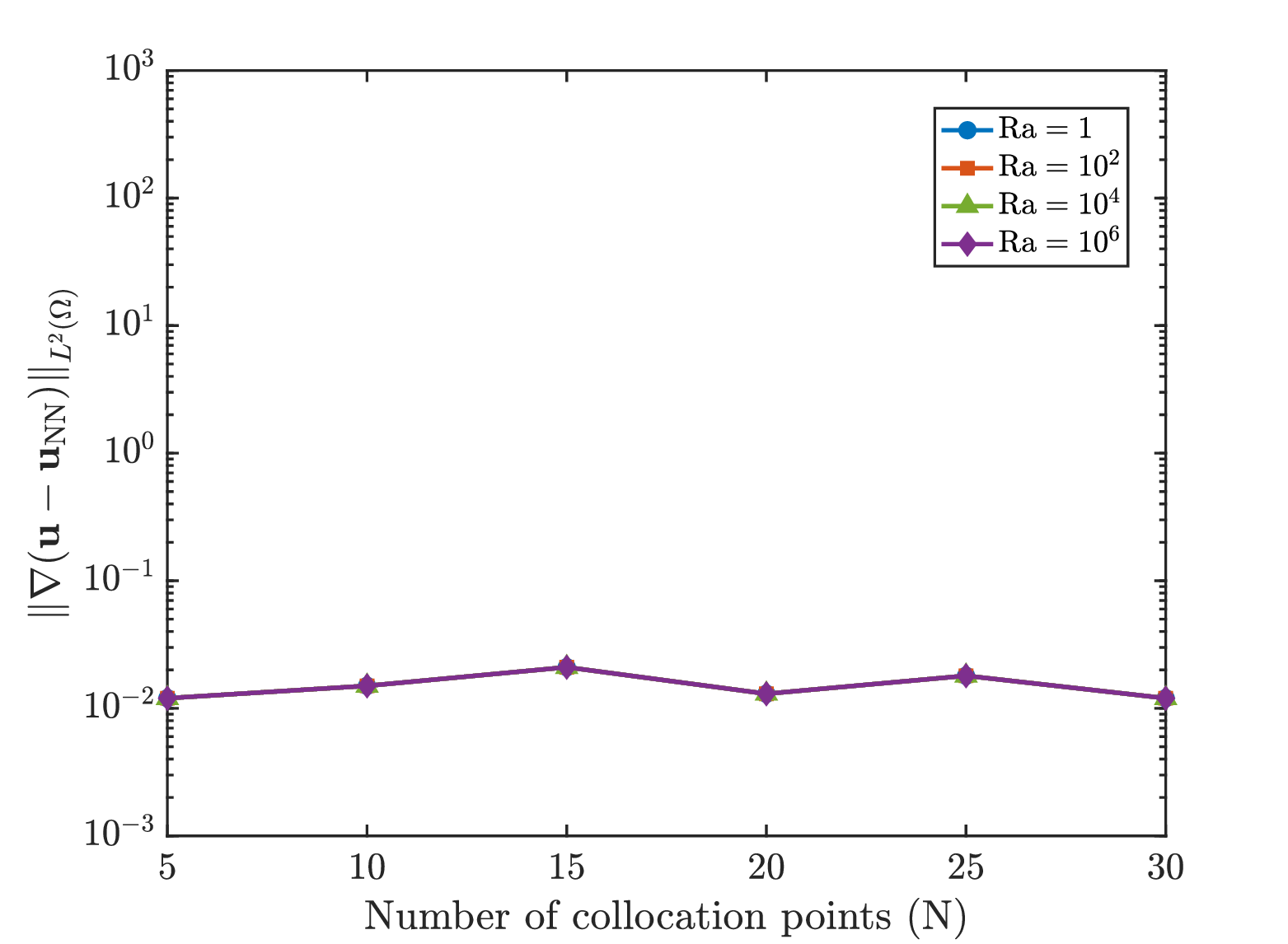}}
\hspace{0.3cm}
\subfloat{\includegraphics[width=0.48\textwidth]{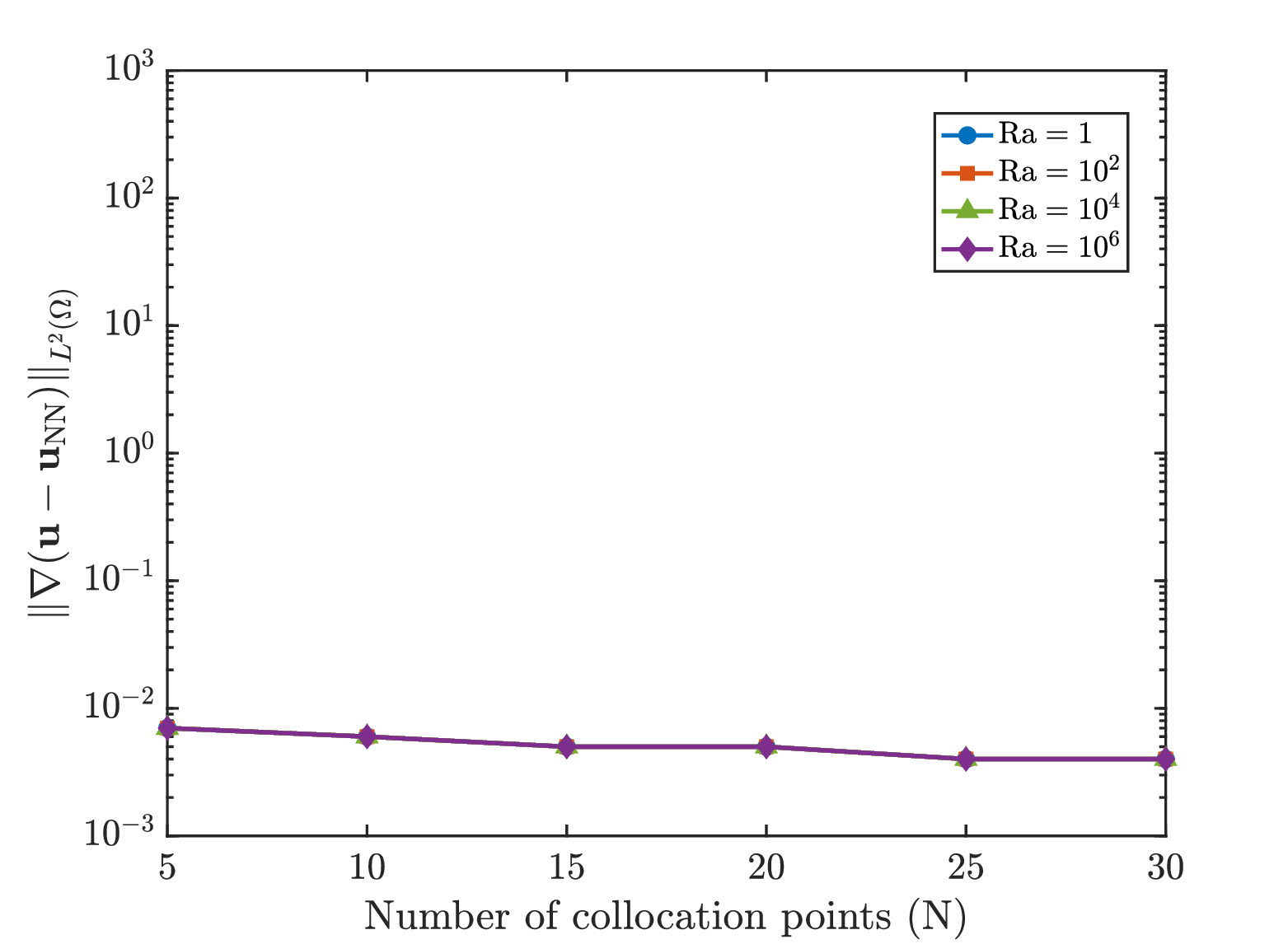}}
\end{figure}

\begin{figure}[H]
\centering
\caption{Convergence plots for the solution using PINNs (left) and CPINNs (right) for pressure-robust formulation.}
\label{fig4}
\subfloat{\includegraphics[width=0.48\textwidth]{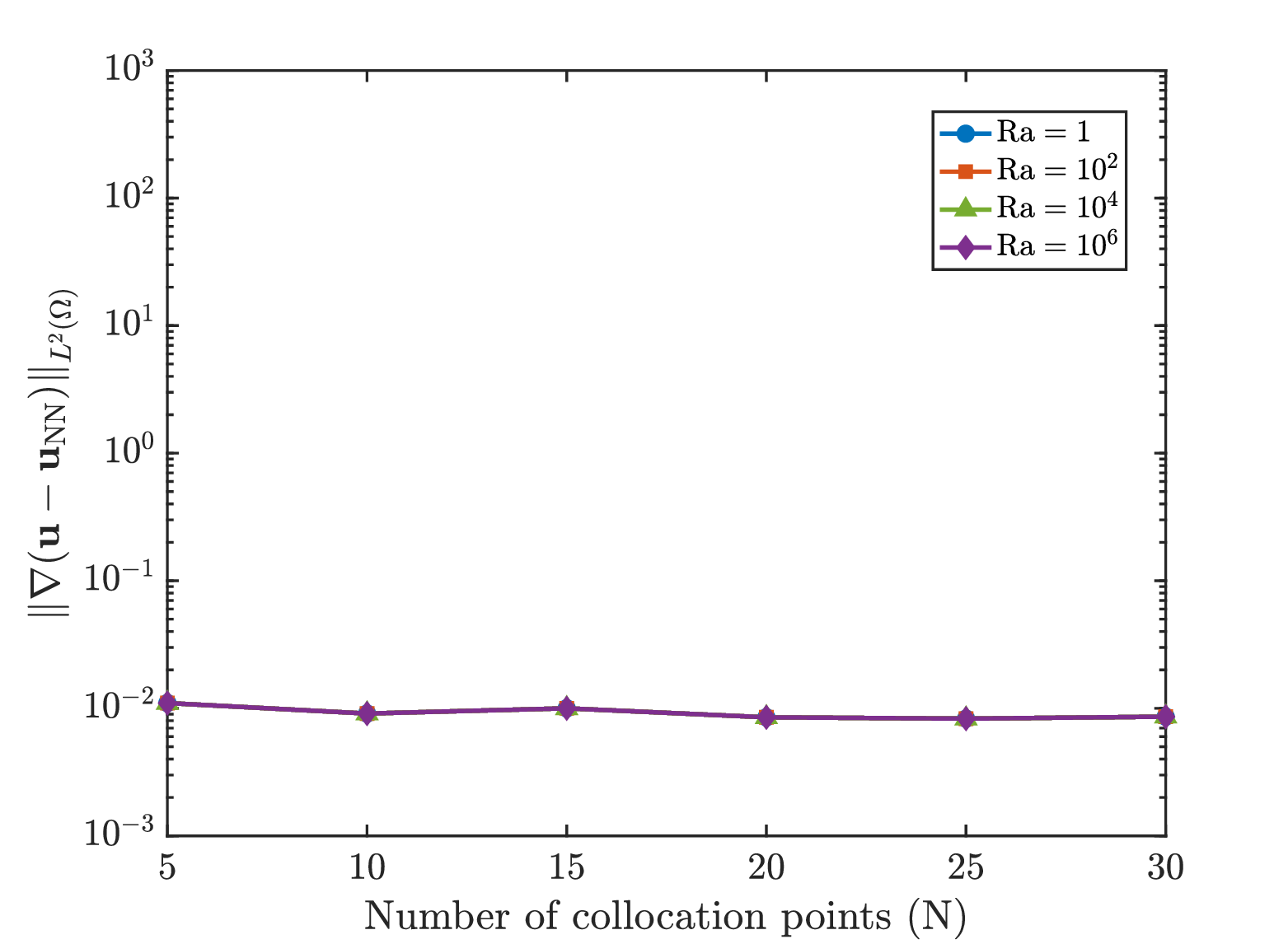}}
\hspace{0.3cm}
\subfloat{\includegraphics[width=0.48\textwidth]{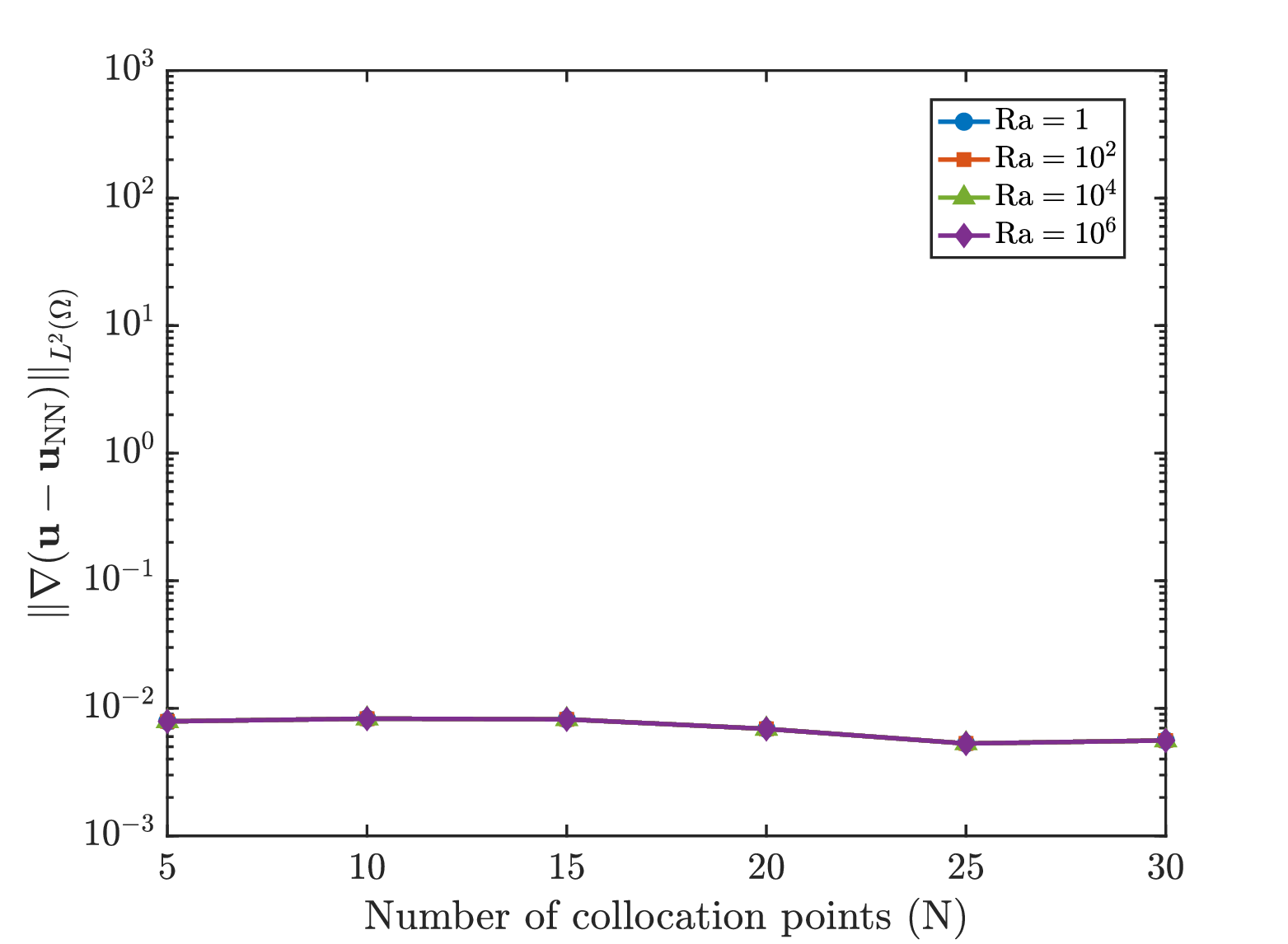}}
\end{figure}

\section{Conclusion}
In this paper, we developed a rigorous framework for PINN approximations of the stationary Oseen equations through the formulation of CPINNs that are derived directly from the stability properties of the continuous problem. By constructing loss functionals that are consistent with the natural functional norms of the PDE, the proposed CPINNs framework provides a principled and theoretically justified approach for learning incompressible flow solutions from sampled data. Within this setting, we propose for the first time both standard CPINNs formulations and a novel pressure-robust CPINNs formulation that removes the influence of gradient forces from the velocity approximation, ensuring that the velocity error depends only on the divergence-free component of the forcing. Using optimal recovery theory under suitable Besov regularity assumptions, we established optimal recovery rates for the velocity in $\boldsymbol H^1(\Omega)$ and the pressure in $L^2(\Omega)$, thereby providing rigorous guarantees for the accuracy of the proposed methods. The resulting framework proposed the ideas for standard CPINNs and pressure-robust CPINNs with theoretical support. Finally, numerical experiments demonstrate the stability, accuracy, and effectiveness of both CPINNs and pressure-robust CPINNs for incompressible flow problems.

\appendix
\section{Proof of Theorem~\ref{thm4.1}}\label{App_A}
	We explain the complete proof here to extend the scalar-valued proof given in \cite{BVPS} to the vector-valued settings. First, we prove the upper bound for $\boldsymbol{f}$ and then the proof will conclude with a lower bound.

\textbf{Proof of the upper bound.} Let us consider the sampling sites $\X$, and we define a vector-valued recovery operator by acting componentwise over the functions $\boldsymbol{f}$ belonging to the unit ball of Besov space $\F$ as
\begin{align}
	\boldsymbol{S}^*(\boldsymbol{f}) = (S^*(f_1), S^*(f_2), \ldots, S^*(f_d)).
\end{align}
Clearly, $\boldsymbol{S}^*(\boldsymbol{f})$ depends only on the data $\{\boldsymbol{f}(x_i)\}_{i=1}^{\tilde m}$, so $\boldsymbol{S}$ is an admissible recovery operator for the vector-valued problem. From the definition of norm $\boldsymbol{H}^{-1}(\Omega)$ and from the scalar bounds from \cite[Theorem~3.1]{BVPS}, we get
\begin{align}\label{normeq}
	\|\boldsymbol{f} - \boldsymbol{S}^*(\boldsymbol{f})\|_{\boldsymbol{H}^{-1}(\Omega)}
	&=
	\Bigl( \sum_{j=1}^d \|f_j - S^*(f_j)\|_{H^{-1}(\Omega)}^2 \Bigr)^{1/2} \\ \notag
	&\lesssim 
	\Bigl( \sum_{j=1}^d C^2\,\tilde m^{-2\alpha_{-1}} \Bigr)^{1/2}
	\lesssim \tilde m^{-\alpha_{-1}}.
\end{align}
Similarly for the case of $d=2$ and $p\leq 1$, we get
\[
\|\boldsymbol{f} -  \boldsymbol{S}^*(f)\|_{ \boldsymbol{H}^{-1}(\Omega)}
\lesssim \tilde m^{-\alpha_{-1}} \log(\tilde m).
\]
Here, the implicit constants depend on $d$ and are independent of $\tilde m$. Taking the supremum over $\boldsymbol{f}\in \F$ and the infimum over all choices of $\X$ and $\boldsymbol{S}$ we get upper bound estimate for $d\ge 3$ or $d=2$, $p>1$ as
\[
R^*_{\tilde m}(\F, \boldsymbol{H}^{-1}(\Omega))
\lesssim \tilde m^{-\alpha_{-1}},
\]
and when $d=2$, $0<p\le 1$, the upper bound is given by 
\[
R^*_{\tilde m}(\F, \boldsymbol{H}^{-1}(\Omega))
\lesssim \tilde m^{-\alpha_{-1}} \log(\tilde m).
\]

\textbf{Proof of the lower bound.}
We will now establish the lower bound for $R^*(\F)_X$. Given arbitrary data sites $\{x_i\}_{i=1}^{\tilde m}$, we construct a function $\boldsymbol{\eta}\in\F$, which satisfies
\begin{align}\label{4.12}
	\boldsymbol{\eta}(x_i) = 0,\quad i = 1,\ldots,\tilde m, \text{ and } \|\boldsymbol{\eta}\|_X\geq c\tilde m^{-\alpha},
\end{align}
where $\alpha$ is the exponent corresponding to $X$. Since both the functions $\boldsymbol{\eta}$ and the zero function satisfy the zero data, we get the lower bound for $R^*(\F)$ from \eqref{4.12}.

In order to construct $\boldsymbol{\eta}$, we define a non-negative, smooth function $\boldsymbol{\varphi}$ on $\mathbb{R}^d$, with support over $\Omega$ and satisfies
\begin{align}\label{4.13}
	\|\varphi_i\|_{H^{-1}(\Omega)} = 1, \quad \varphi_i(x)\geq 1/2,\ \text{for } x\in\Omega_0:=[1/4,3/4]^d,\quad i = 1,2,\ldots d,
\end{align}
and assume $\varphi_i(x)= \phi_i(x_1)\cdots\phi_i(x_d)$ with a univariate function $\phi_i$ holding all these properties. Among all such functions choose $\varphi_i = \varphi_{i,s,p,q}$ with minimal Besov norm
\begin{align}
	\|\varphi_i\|_{B^s_{pq}(\Omega)} := M_{i,s,p,q}.
\end{align}
For any cube $I\subset\Omega$ with smallest vertex $\xi_{I}$ and sidelength $l_I$, we define the rescaled bump function
\begin{align}\label{4.15}
	\boldsymbol{\varphi}_I(x):= l_I^{s-d/p}\boldsymbol{\varphi}(l_I^{-1}(x-\xi_I)),
\end{align} 
which is supported in $I$ and vanishes on its boundary. A direct scaling computation shows $\|\varphi_{i,I}\|_{B^s_{pq} (\Omega)}= M_{i,s,p,q}$ for all cubes and for all $i=1,\ldots ,d$.

We assume $m=2^{kd}$, and for the given data site $\X_{\tilde m}=\{\boldsymbol x_1,\ldots, \boldsymbol x_{\tilde m}\}\subset\Omega$, consider the tensor-product grid $G_{k+2,2}$ with mesh size $2^{-(k+2)}$. Thus, the grid contains $4^d\tilde m$ cubes, so atleast $(4^d-1)\tilde m$ of them contain no data point in their interior. Let $\I = \I(\X_{\tilde m})$ denote the family of such cubes. Then
\begin{align}
	|\I|\geq (4^d-1)\tilde m,\quad |I|\simeq2^{-d(k+2)}\geq c \tilde m^{-1},\ \forall I\in\I.
\end{align}
For each empty cube $I\in \I$ define the normalized bump function $\boldsymbol{\eta}_I := M^{-1}\boldsymbol{\varphi}_I$. Then we get each $\boldsymbol{\eta}_I\in\F$ and for all data sites $\boldsymbol{\eta}_I(x_i)=0$. These functions $\boldsymbol{\eta}_I$ used to derive the corresponding lower bounds for $R^*(\F)_X$ for target space $X$. 

For $X = \boldsymbol{H}^{-1}(\Omega)$ the argument splits depending on whether $p\leq \gamma$ or $p\geq \gamma$, where $\gamma$ is given in theorem.

\textbf{Case $p\leq \gamma$ :} For any non-empty cube $I\in \I$, we define $\boldsymbol{\eta} = \boldsymbol{\eta} _I$. From the scaling of $\boldsymbol{\varphi}_I$, for every $i=1,\ldots,d$, we  have 
\begin{align}
	\eta_i(x) = M^{-1}\varphi_{i,I}\geq M^{-1}l_{I}^{s-d/p}\geq c\tilde m^{-\frac{s}{d}+\frac{1}{p}},\ x\in I_{0}:= \left[\xi_I-\frac{l_I}{4}, \xi_I+\frac{l_I}{4}\right]^d,
\end{align}
where $I_0$ is the subcube of $I$ corresponding to $\Omega_0$ with $|I_0|\geq c\tilde m^{-1}$. In order to estimate $\|\boldsymbol{\eta}\|_{\boldsymbol{H}^{-1}(\Omega)}$, we define a function $\boldsymbol{v}\in \boldsymbol{H}^1_0(\Omega)$, with support over $I$, $\|v_i\|_{H^1_{0}(\Omega)} =1$ for $i=1,\ldots ,d$ and satisfy 
\begin{align}
	v_i(x )= c\,l_I^{1-d/2}\varphi_i(l_I^{-1}(x-\xi_I))\geq c\,l_I^{1-d/2}\geq c\tilde m^{-\frac{1}{d}+\frac{1}{2}} \quad x\in I_0.
\end{align}
Hence, the corresponding estimate satisfies 
\begin{align}
	\|\boldsymbol{\eta}\|_{\boldsymbol{H}^{-1}(\Omega)}&\geq \left[\sum_{i=1}^{d}\left(\int_{\Omega}^{}\eta_i\, v_i\,dx\right)^2\right]^{\frac{1}{2}} =  \left[\sum_{i=1}^{d}\left(\int_{I_0}^{} \eta_i\, v_i\,dx\right)^2\right]^{\frac{1}{2}}\\ \notag &\gtrsim \,  \tilde m^{-\frac{s}{d}+\frac{1}{p}}\tilde m^{-\frac{1}{d}+\frac{1}{2}}\,|I_0|\gtrsim \tilde m^{-\frac{s}{d}+\frac{1}{p}-\frac{1}{\gamma}}.
\end{align}
Thus the argument $R^*_m(\F)\gtrsim\,\tilde m^{-\alpha_{-1}}$ holds for $p\leq \gamma$.

\textbf{Case $p\geq \gamma$ :} Since $\F = U(\boldsymbol B^s_{\infty,q}(\Omega))\subset U(\boldsymbol B^s_{pq}(\Omega))$, We restrict our attention to the case $p=\infty$. Define the global bump function
\begin{align}
	\boldsymbol{\eta} = \kappa\sum_{I\in\I}^{}\boldsymbol{\eta}_I
\end{align}
where $\kappa$ is a constant such that $\boldsymbol{\eta}\in \F$. Each $\eta_{i,I}\geq 2^{-ks}$ on its corresponding inner cube for $\tilde m= 2^{kd}$. Let us define a vector-valued function $\boldsymbol{v}(x)=c\,\boldsymbol{\varphi}(x)$ such that $v_i\geq c>0$ on $\Omega_0$ and $\|v_i\|_{H^1_0(\Omega)} = 1$ for all $i=1,\ldots,d$. Define $\I_0 = \{I\in\I: I\subset \Omega_0\}$, where the grid inside $\Omega_0$ has $2^d\tilde m$ cubes and $|\I_0|\geq (2^d-1)\tilde m$. Thus, the following estimate holds
\begin{align}
	\|\boldsymbol{\eta}\|_{\boldsymbol{H}^{-1}(\Omega)}&\geq  \left[\sum_{i=1}^{d}\left(\int_{\Omega}^{}\eta_i\, v_i\,dx\right)^2\right]^{\frac{1}{2}}  \geq  \left[\sum_{i=1}^{d}\left(\int_{I_0}^{} \eta_i\, v_i\,dx\right)^2\right]^{\frac{1}{2}} \\\notag
	&\geq  \left[\sum_{i=1}^{d}\left(\int_{\Omega_0}^{}\eta_i\,dx\right)^2\right]^{\frac{1}{2}} \geq \left[\sum_{i=1}^{d}\left(\kappa\sum_{I\in\I_0}^{}\int_{I}^{}\eta_{i,I}\,dx\right)^2\right]^{\frac{1}{2}} \\\notag
	&\geq c\,2^{-ks}2^{-kd}(2^d-1)\tilde m\gtrsim \tilde m^{-s/d}.
\end{align}
Since $1/p-1/\gamma<0$, the exponent matches the exponent $\alpha_{-1}$ given in theorem which proves $R^*_{X}(\F)$ for the case $p>\gamma$. This establishes the lower bound required for $X = \boldsymbol H^{-1}(\Omega)$, and thus concludes the proof.

\section{Proof of Theorem~\ref{thm4.2}}\label{App_B}
	For any $\boldsymbol{g}\in\G$, we choose $\boldsymbol{v}\in \overline{\B}$ with $\operatorname{Tr}(\boldsymbol{v})=\boldsymbol{g}$. After applying the trace norm definition, we get
\begin{align}
	\|\boldsymbol{g}-\overline{\boldsymbol{S}}_k(\boldsymbol{g})\|_{\boldsymbol{H}^{1/2}(\Gamma)}\leq \|\boldsymbol{v}-\boldsymbol{S}_k\|_{\boldsymbol{H}^1(\Omega)}=\left[\sum_{i=1}^{d}\|v_i-S_{i,k}\|^2_{H^1(\Omega)}\right]^{\frac{1}{2}}.
\end{align}
For every scalar component $v_i\in \B_i$, the scalar version of the approximation result proved in~\cite[Theorem~3.1]{BVPS}, yields the upper bound
\begin{align}
	\|v_i-S_{i,k}\|_{H^1(\Omega)} \leq C(r2^k)^{-\bar s+1}.
\end{align}
where the constant $C$ is independent from $i$. Taking sum for $i=1,\ldots,d$ with $\bar{m}\asymp2^{k(d-1)} $ gives the bound
\begin{align}
	\|\boldsymbol{g}-\overline{\boldsymbol{S}}_k(\boldsymbol{g})\|_{\boldsymbol{H}^{1/2}(\Gamma)}\leq \left[\sum_{i=1}^{d}\|v_i-S_{i,k}\|^2_{H^1(\Omega)}\right]^{\frac{1}{2}}  \leq C(r2^k)^{-\bar s+1} \leq C\, \bar m^{-\beta}.
\end{align}
As this holds for every $\boldsymbol{g}\in \G,$ we get the desired upper bound of $R^*_{\bar{m}}(\G)$ for the vector-valued case.   

Next, we will prove the lower bound. We now fix $\bar m = 2^{k(d-1)}-1$  and consider any set of boundary sampling points $\Y = \{\boldsymbol{z}_1,\ldots,\boldsymbol{z}_{\bar{m}}\}$. We consider a boundary face $F = \{x\in\Omega: x_1=0\}$ and define a dyadic decomposition $\D_k(F)$ of  $F$. Since $\D_k(F)$ contains $2^{k(d-1)}$ cubes for $\bar m$ sample points, there exist a cube $\overline{J}\in\D_k(F)$ that contains no data site in its interior. Let $J\subset \Omega$ be the $d-$dimensional dyadic cube with $\overline{J}$ as a face. Let $\boldsymbol{\varphi}_J$ be the scaled bump function defined in \eqref{4.15} and for $M = M_{\bar s, \bar p, \bar q}$ set
\begin{align}
	\boldsymbol{v}(\boldsymbol{x}) := M^{-1}\boldsymbol{\varphi}_J (\boldsymbol{x}-(2^{-k-1},0,\ldots,0)), \quad \text{for } \boldsymbol{x}\in \Omega.
\end{align}
Then $\boldsymbol{v}\in U(\boldsymbol{B}^{\bar{s}}_{\bar{p}\bar{q}})$ and its boundary trace function $\boldsymbol{\eta} := \operatorname{Tr}(\boldsymbol{v})$ belongs to $\G$ and vanishes at $\Y$. Let $\boldsymbol{\varphi}$ be the bump function defined in \eqref{4.13} and $\boldsymbol{\eta}_0$ be its trace on the hyperplane $x_1=1/2$. Define $M_0:= |\boldsymbol{\eta}_0|_{\boldsymbol H^{1/2}(\Gamma)}$ and trace function $\boldsymbol{\eta}$ as a translated and rescaled function of $\boldsymbol{\eta}_0$. Thus, for $\bar m\asymp 2^{k(d-1)}$, a change of variables in function $\boldsymbol{\eta}$ yields the estimate
\begin{align}
	|\boldsymbol{\eta}|_{\boldsymbol{H}^{1/2}(\Gamma)} =\left[ \sum_{i=1}^{d}|\eta_i|^2_{H^{1/2}(\Gamma)}\right]^{\frac{1}{2}}
	=M_0M^{-1}2^{-k(s-d/p)} 2^{k/2} 2^{-k(d-1)/2}\gtrsim \bar m^{-\beta},
\end{align}
for a positive constant independent of $\bar m$. Since the data sites were arbitrary, this gives the lower bound $R^*_{\bar m}(\G)_{\boldsymbol{H}^{1/2}(\Gamma)} \geq c\,\bar m^{-\beta}$ and thus completes the proof.

\bibliographystyle{siamplain}
\bibliography{reference_2n}
\end{document}